\documentclass[11pt,a4paper,oneside]{article}

\usepackage[latin1]{inputenc}
\usepackage{amsmath,amsthm}
\usepackage[english]{babel}
\usepackage{latexsym}
\usepackage{amsfonts}
\usepackage{amssymb}
\usepackage{bigints}
\usepackage{euscript}
\usepackage{color}
\usepackage{longtable}
\usepackage{tabularx}
\usepackage{pdfsync}
\usepackage{pstricks,pst-plot,pst-node,epsfig}
\usepackage{hyperref}
\usepackage{pgfplots}
\usepackage{multicol}
\usepackage{mathtools}

\catcode`\@=11
\def\newremark#1{\@ifnextchar[{\@orem{#1}}{\@nrem{#1}}}
\def\@nrem#1#2{%
\@ifnextchar[{\@xnrem{#1}{#2}}{\@ynrem{#1}{#2}}}
\def\@xnrem#1#2[#3]{\expandafter\@ifdefinable\csname #1\endcsname
{\@definecounter{#1}\@addtoreset{#1}{#3}%
\expandafter\xdef\csname the#1\endcsname{\expandafter\noexpand
    \csname the#3\endcsname \@remcountersep \@remcounter{#1}}%
\global\@namedef{#1}{\@rem{#1}{#2}}\global\@namedef{end#1}{\@endremark}}}
\def\@ynrem#1#2{\expandafter\@ifdefinable\csname #1\endcsname
{\@definecounter{#1}%
\expandafter\xdef\csname the#1\endcsname{\@remcounter{#1}}%
\global\@namedef{#1}{\@rem{#1}{#2}}\global\@namedef{end#1}{\@endremark}}}
\def\@orem#1[#2]#3{\expandafter\@ifdefinable\csname #1\endcsname
    {\global\@namedef{the#1}{\@nameuse{the#2}}%
\global\@namedef{#1}{\@rem{#2}{#3}}%
\global\@namedef{end#1}{\@endremark}}}
\def\@rem#1#2{\refstepcounter
      {#1}\@ifnextchar[{\@yrem{#1}{#2}}{\@xrem{#1}{#2}}}
\def\@xrem#1#2{\@beginremark{#2}{\csname the#1\endcsname}\ignorespaces}
\def\@yrem#1#2[#3]{\@opargbeginremark{#2}{\csname
         the#1\endcsname}{#3}\ignorespaces}
\def\@remcounter#1{\noexpand\arabic{#1}}
\def\@remcountersep{.}
\def\@beginremark#1#2{\rm \trivlist \item[\hskip \labelsep{\bf #1\ #2}]}
\def\@opargbeginremark#1#2#3{\rm \trivlist
        \item[\hskip \labelsep{\bf #1\ #2\ (#3)}]}
\def\@endremark{\endtrivlist}
\catcode`\@=12
\newcommand{\biindice}[3]%
{

\begin{array}[t]{c}
#1\\
{\scriptstyle #2}\\
{\scriptstyle #3}
\end{array}

}
\def\a{\alpha}

\def\s{\sigma}

\def\l{\lambda}

\newcommand{\R}{\mathbb R}

\numberwithin{equation}{section}

\theoremstyle{definition}

\theoremstyle{plain}
\newtheorem{theorem}{Theorem}[section]
\newtheorem{proposition}{Proposition}[section]
\newtheorem{lemma}{Lemma}[section]

\newtheorem{remark}{Remark}[section]

\title{\vskip-2.5cm
{\textbf{A Maupertuis-type principle in relativistic mechanics and applications}}
\thanks{Work written under the auspices of the Gruppo Nazionale per l'Analisi Matematica, la Probabilit\`{a}
e le loro Applicazioni (GNAMPA) of the Istituto Nazionale di Alta Matematica (INdAM). The third author has been also supported by the Ministry of
Science, Technology and Universities of Spain, under Research Grant PGC2018-097104-B-100, and by contract CT42/18-CT43/18 of Complutense University of Madrid.
}
}

\author{
\sc
Alberto Boscaggin
\\
\small
Universit\`{a} degli Studi di Torino
\\
\small
Dipartimento di Matematica ``Giuseppe Peano"
\\
\small  Via Carlo Alberto 10, 10123 Torino, Italy. 
\\
\small E-mail: {\tt alberto.boscaggin@unito.it}
\medskip
\\
\sc
Walter Dambrosio
\\
\small
Universit\`{a} degli Studi di Torino
\\
\small
Dipartimento di Matematica ``Giuseppe Peano"
\\
\small  Via Carlo Alberto 10, 10123 Torino, Italy. 
\\
\small E-mail: {\tt walter.dambrosio@unito.it}
\medskip
\\
\sc Eduardo Mu\~{n}oz-Hern\'andez
\\
\small
Universidad Complutense de Madrid
\\
\small
Instituto de Matem\'{a}tica Interdisciplinar (IMI)
\\
\small
Departamento de An\'alisis Matem\'atico y Matem\'atica Aplicada
\\
\small  Plaza de las Ciencias 3, 28040 Madrid, Spain
\\
\small E-mail: {\tt eduardmu@ucm.es }
}

\date{}

\numberwithin{equation}{section}

\begin{document}

\maketitle

{

\begin{abstract}
We provide a Maupertuis-type principle for the following system of ODE, of interest in special relativity:
$$
\frac{\rm d}{{\rm d}t}\left(\frac{m\dot{x}}{\sqrt{1-|\dot{x}|^2/c^2}}\right)=\nabla V(x),\qquad x\in\Omega \subset \mathbb{R}^n,
$$
where $m, c > 0$ and $V: \Omega \to \mathbb{R}$ is a function of class $C^1$. As an application, we prove the existence of multiple periodic solutions with prescribed energy for a relativistic $N$-centre type problem in the plane.
\end{abstract}

}


\section{Introduction}\label{sec1}

The well-known Maupertuis principle of classical mechanics (see, for instance, \cite[page 247]{Arn89} or \cite[equation (8.89)]{GPS02}) states that solutions of the system of ODEs
\begin{equation}\label{intro_class}
\ddot x = \nabla V(x), \qquad x \in \Omega \subset \mathbb{R}^n,
\end{equation}
lying on a fixed energy level
$$
\frac{1}{2} \vert \dot x(t) \vert^2 - V(x(t)) \equiv h \in \mathbb{R},
$$
correspond, up to a time-reparameterization, to critical points $q$ of the energy functional 
$$
\int_0^1 \vert \dot q(\sigma) \vert^2 (V(q(\sigma)) + h) \,d\sigma, 
$$
that is, to geodesic curves on the Hill's region $\{ q \in \Omega \, : \, V(q) + h > 0\}$, when endowed with the Riemannian metric $g_{ij}(q) = (V(q) + h) \delta_{ij}$ (the so-called Jacobi metrix; here $\delta_{ij}$ is the usual Kronecker delta).

Taking advantage of this variational principle, in the last decades several existence and multiplicity results for solutions with \emph{prescribed energy} of equation \eqref{intro_class} have been proved by the use of the methods of critical point theory. In particular, due to its mechanics relevance, cases when the potential $V$ is singular have been the object of extensive investigations: we refer the reader to the book \cite{ACZ93} for a recap of most of the results obtained until the early nineties, and to \cite{BDT17,Cas17,FT00} and the references therein for some more recent contributions. 

In this paper, we look for a corresponding variational principle for the equation
\begin{equation}\label{intro_rel}
\frac{\rm d}{{\rm d}t}\left(\frac{m\dot{x}}{\sqrt{1-|\dot{x}|^2/c^2}}\right)=\nabla V(x),\qquad x\in\Omega \subset \mathbb{R}^n,
\end{equation}
where $m, c$ are positive constants: here, according to the underlying Lagrange structure (see the beginning of Section \ref{sec2} for more details), preservation of energy means that
\begin{equation}\label{ene-intro}
mc^2 \left(\frac{1}{\sqrt{1-|\dot{x}(t)|^2/c^2}}-1\right)-V(x(t)) \equiv h.
\end{equation}
As well known (see, for instance, \cite[Chapter 7]{GPS02} or \cite[Chapter 33]{Gre04}), equation \eqref{intro_rel} finds its motivation from 
special relativity, providing an approximate model for the motion of a particle $x$ of mass $m$, moving at speed close to that of light, denoted here\footnote{Even if in the mathematical literature it is common to use the normalization assumption $m = c = 1$, as in \cite{BDF21,BDP-pp}we find useful to keep track of both these constants, in order to have a clearer physical interpretation of the results obtained, as well as to 
provide some considerations about the non-relativistic limit $c \to +\infty$.}  by $c$.
Just to mention some few references in this direction, we refer to the papers \cite{AB71,Boy04,MP06}, dealing with equation \eqref{intro_rel} in the case when $V$ is a Kepler or Coulomb potential.

From a genuinely mathematical point of view, equation \eqref{intro_rel} has been widely investigated in the last decades, even in the more general case when the right-hand side depends explicitly on time, and possibly does not have a gradient-type structure. Nowadays, an enormous amount of existence and multiplicity results, obtained with both variational and topological methods, is available for the associated boundary value problems: we refer to \cite{Maw13} for a ten-year-old survey on the topic, and to \cite{ABT20,BDP-pp} for some recent papers. In spite of this extensive bibliography, however, to the best of our knowledge there are no general results dealing with solutions of equation \eqref{intro_rel} having prescribed energy $h \in \mathbb{R}$ and, in fact, the underlying Maupertuis principle seems to be missing in the literature.

Our paper aims at filling this gap and is organized as follows. In Section \ref{sec2} we provide the relativistic Maupertuis-type principle, cf. Proposition \ref{pr-var}. Roughly, it states the following: solutions to equation \eqref{intro_rel} having prescribed energy \eqref{ene-intro} correspond, up to a time-reparameterization, to critical points $q$ of the functional
\begin{equation}\label{intro-relfun}
\mathcal{E}_h(q) = \int_0^1 \vert \dot q(\sigma) \vert^2 (Z_h(q(\sigma)) + 2hm) \,d\sigma, 
\end{equation}
where
\begin{equation}\label{intro-pot}
Z_h(q) = 2mV(q)+\dfrac{1}{c^2}\, (V(q)+h)^2.
\end{equation}
Thus, solutions of \eqref{intro_rel} are again geodesics on the Hill's region $\{ q \in \Omega \, : \, V(q) + h > 0\}$, 
but relativistic effects result into a perturbation of the usual Jacobi metric of classical mechanics: indeed, notice that in the non-relativistic limit $c \to +\infty$ the relativistic metric $g_{ij}(q) = (Z_h(q) + 2hm)\delta_{ij}$ reduces, up to a constant factor, to the usual one. We emphasize that, in the functional \eqref{intro-relfun}, no bounds on the modulus of $\dot q$ are needed. At first glance, this might seem quite surprising, since equation \eqref{intro_rel} implicitly requires $\vert \dot x \vert < c$; however, it is coherent with the fact that the fixed energy problem is of purely geometrical nature: the kinematic property $\vert \dot x \vert < c$, indeed, will be automatically obtained by the time-reparameterization. 

At the end of Section \ref{sec2}, an equivalent formulation of the variational principle is given in Proposition \ref{pr-maup}, using instead of $\mathcal{E}_h$ the so-called Maupertuis functional
\[
\mathcal{M}_h(q) = \int_0^1 \vert \dot q(\sigma) \vert^2\,d\sigma \int_0^1  (Z_h(q(\sigma)) + 2hm) \,d\sigma,
\]
which, as well known, is more manageable when using critical point theory, cf. \cite[Remark 4.2]{ACZ93}.

As an application of the Maupertuis principle, in Section \ref{sec3} we investigate the existence of periodic solutions, with prescribed energy, for the equation
\begin{equation}
\label{eq-ncariche_intro}
\frac{\rm d}{{\rm d}t}\left(\frac{m\dot{x}}{\sqrt{1-|\dot{x}|^2/c^2}}\right)=- \sum_{i=1}^N \frac{\kappa_i (x-\sigma_i)}{\vert x - \sigma_i \vert^{\alpha+2}} + \nabla W(x), \qquad x \in \mathbb{R}^2,
\end{equation}
where $\sigma_1,\ldots,\sigma_N$ are distinct points of $\mathbb{R}^2$, $\kappa_i > 0$, $\alpha > 1$ and $W$ is meant as a perturbation term as $\vert x \vert \to +\infty$. Such an equation can be interpreted as a (perturbed) relativistic $N$-centre problem, providing the motion of a charge $x$ in the electric field generated by $N$-fixed charges $\sigma_1,\ldots,\sigma_N$, cf. \cite{ABT20}.
By minimization arguments for the associated energy functional, and taking advantage of the topology of the punctured plane (see \cite{FT00} for a somehow related result in the non-relativistic setting), we prove the existence of multiple periodic solutions of equation \eqref{eq-ncariche_intro}, with prescribed homotopy class and prescribed energy \eqref{ene-intro}, for any $h > 0$ arbitrarily fixed, cf. Theorem \ref{thmain} for the precise statement. We point out that, for the whole argument, the restriction $\alpha > 1$ is needed: this has to be interpreted as a relativistic version of the usual strong-force condition of classical mechanics, see \cite{Gor75}. Notice however that, while in the classical case one is led to assume $\alpha \geq 2$, here a weaker assumption is enough: this comes from the fact that the potential $Z_h$ given in \eqref{intro-pot} behaves, near the singularity, as the square of the potential $V$.
There is however a subtle point, since the physically relevant case $\alpha = 1$, which would give rise to a potential $Z_h$ of order $2$ near the singularity, cannot be treated within our approach: indeed, as explained in Remark \ref{remstrong}, a further condition for $Z_h$ is needed, requiring eventually $\alpha > 1$. At first sight, Theorem \ref{thmain} thus seems to be in sharp contrast with the results obtained in \cite{BDP-pp} for the fixed period problem: as a matter of fact, however, the difficulty comes from the specific structure of the Maupertuis functional, which, compared with the action functional, requires an independent control of the kinetic and potential parts, cf. \cite[Remark 5.4]{ACZ93}. We stress that this difficulty is not only of technical nature: indeed, as pointed
out in Remark \ref{rem_alfa}, in the case $\alpha = 1$ a different range for the admissible energies is expected, giving rise to a different geometry for the associated functional. Of course, due do its mechanical significance, the case $\alpha = 1$ deserves further investigation in the future.

Finally, in Section \ref{sec4} we use phase-plane analysis (for the model problem \eqref{eq-ncariche_intro} for $N = 1$ and $W \equiv 0$) to investigate the optimality of the condition $h > 0$ appearing in Theorem \ref{thmain}, finally proving that it is sharp when $\alpha \geq 2$.

\section{A variational principle for prescribed energy relativistic problems}\label{sec2}

Let us consider the system of ODEs
\begin{equation}
\label{2.1}
\frac{\rm d}{{\rm d}t}\left(\frac{m\dot{x}}{\sqrt{1-|\dot{x}|^2/c^2}}\right)=\nabla V(x),\qquad x\in\Omega,
\end{equation}
where $m, c > 0$, $\Omega \subset \R^n$ is an open set and $V:\Omega \subset \R^n\to \R$ is a function of class $C^1$.
This equation can be meant as the Euler-Lagrange equation 
\[
\frac{\rm d}{{\rm d}t}\left(\frac{\partial L}{\partial \dot{x}}\right)=\frac{\partial L}{\partial x}
\]
for the Lagrangian 
$$
L(x,\dot{x})=-mc^2\sqrt{1-\dfrac{|\dot{x}|^2}{c^2}}+V(x).
$$ 
We can then pass to an equivalent Hamiltonian formulation defining the dual variable 
\begin{equation}
\label{dual}
p=\frac{m\dot{x}}{\sqrt{1-\dfrac{|\dot{x}|^2}{c^2}}};
\end{equation}
accordingly, by Legendre transformation, the associated Hamiltonian function turns out to be 
\begin{equation}
\label{hamilt}
H(x,p)=mc^2\sqrt{1+\frac{|p|^2}{m^2c^2}}-V(x).
\end{equation}
As well known, the Hamiltonian $H$ is a constant of motion. The \emph{relativistic energy} is nothing but the Hamiltonian $H$ written in terms of $x$ and $\dot x$; using \eqref{dual}, we thus find 
$$
E(x,\dot x) = \frac{mc^2}{\sqrt{1-\dfrac{|\dot{x}|^2}{c^2}}}-V(x)
$$
and, so, for any solution $x$ of \eqref{2.1} there exists $\mathfrak{E} \in \mathbb{R}$ such that
\begin{equation}\label{defenergy}
\frac{mc^2}{\sqrt{1-\dfrac{|\dot{x}(t)|^2}{c^2}}}-V(x(t)) \equiv \mathfrak{E}. 
\end{equation}
Let us notice that \eqref{defenergy} implies that, in the case of potential $V$ such that $V(x)\to 0$ as $|x|\to +\infty$, the particle rest energy is $mc^2$. Thus, it is convenient to write
$$
\mathfrak{E}=h+mc^2,
$$
in such a way that $h$ is the difference from the relativistic energy and particle rest energy, cf. with \eqref{intro_rel} of the Introduction. Incidentally, let us notice that (see Remark \ref{rem-nonrel1} for more details) that the \textit{non-relativistic} limit of $h= \mathfrak{E}-mc^2$ as $c\to +\infty$ is the usual classical mechanical energy.

Let us also observe that solutions $x$ with energy $\mathfrak{E} = h + mc^2$ satisfy
\[
V(x(t))+h>0,\quad \forall \ t,
\]
and so are confined in the
\begin{equation} \label{eq-defomegah}
\Omega_h=\{x\in \Omega:\ V(x)+h >0 \},
\end{equation}
which is then the analogous of the Hill's region in classical mechanics.
\medbreak
The aim of this section is to introduce a Maupertuis-type functional whose critical points are solutions of \eqref{2.1} with prescribed energy: in order to do this, we follow the same ideas of the non-relativistic case, relating the action, length, energy and Maupertuis functionals.

Motivated by the application of Section \ref{sec3}, from now on we focus on the case of periodic boundary conditions: it will be clear, however, that the whole discussion could be performed with very minor changes if other boundary value problems (like, for instance, the two-point one) are taken into account.

\subsection{The heuristic derivation: from action to length}

Let us begin our discussion with an heuristic approach leading to a length functional starting from the action functional, without focusing on precise statements and considerations. 

\smallbreak
First of all, let us fix $h\in \R$ such that 
\begin{equation} \label{eq-dominiononvuoto}
\Omega_h\neq \emptyset,
\end{equation}
where $\Omega_h$ is defined in \eqref{eq-defomegah}. For a given $T>0$, we consider the space $W^{1,\infty}_{T}=\{x\in W^{1,\infty}([0,T])\,:\, x(0)=x(T)\}$ and the set
\[
\Lambda_h=\{x\in W^{1,\infty}_{T}: \ x(t)\in \Omega_h,\ |\dot{x}(t)|\leq c,\ \forall \ t\in [0,T]\}.
\] 
We then define the \emph{action functional} $\mathcal{A}_T$ by
\[
\mathcal{A}_T(x)=\int_0^T \left(-mc^2\sqrt{1-\frac{|\dot{x}(t)|^2}{c^2}}+V(x(t))+h+mc^2\right)\; dt,
\]
for every $x\in \Lambda_h$. As well known, its critical points (in an appropriate sense, cf. \cite{ABT20,BDP-pp}) are $T$-periodic solutions of \eqref{2.1}. Now, for every increasing diffeomorphism $t:[0,1]\rightarrow[0,T]$,
let us define 
\[q(\s)=x(t(\s)),
\] 
for every $\s\in [0,1]$. Hence, we plainly deduce that the action functional reads as
\begin{equation*}
\begin{array}{l}
\displaystyle \int_0^1 t'(\s)\left(-mc^2\sqrt{1-\frac{|q'(\s)|^2}{(t'(\s))^2c^2}}+V(q(\s))+h+mc^2\right)\; d\s\,.
\end{array}
\end{equation*}
For every $\s\in [0,1]$, let
\begin{equation} \label{eq-heur1}
b=V(q(\s))+h+mc^2,\quad z=\dfrac{|q'(\sigma)|^2}{c^2}
\end{equation}
and
\[
f(\zeta)=\zeta \left(-mc^2\sqrt{1-\frac{z}{\zeta^2}}+b\right),
\]
for every $\zeta$ with $\zeta\geq \sqrt{z}$. Let us observe that assumption \eqref{eq-dominiononvuoto} implies that $b>mc^2$; hence, a simple computation shows that $f$ has the unique minimum point
\[
\zeta_0=b\, \sqrt{\frac{z}{b^2-m^2c^4}}=\dfrac{V(q(\s))+h+mc^2}{\sqrt{2mV(q(\s))+\dfrac{1}{c^2}\, (V(q(\s))+h)^2+2hm}}\, \dfrac{|q'(\s)|}{c^2},
\]
with
\begin{equation} \label{eq-heur2}
\begin{array}{ll}
\displaystyle f(\zeta_0)&\displaystyle =b\, \sqrt{\frac{z}{b^2-m^2c^4}} \left(-mc^2\sqrt{1-\frac{z}{\zeta_0^2}}+b\right)=\dfrac{{|q'(\s)|}}{c}\, \sqrt{b^2-m^2c^4}\\
&\\
&\displaystyle = |q'(\sigma)|\sqrt{2mV(q(\s))+\dfrac{1}{c^2}\, (V(q(\s))+h)^2+2hm},
\end{array}
\end{equation}
by \eqref{eq-heur1}. Now, let us define
\begin{equation} \label{def-Z}
Z_h(x)=2mV(x)+\dfrac{1}{c^2}\, (V(x)+h)^2,\quad \forall \ x\in \Omega_h,
\end{equation}
and observe that
$$
Z_h(x) + 2hm > 0, \quad \forall \ x \in \Omega_h.
$$
Accordingly, we can define the \emph{length functional} 
\begin{equation} \label{eq-deflungh}
\mathcal{L}_h(q)=\int_0^1|q'(\sigma)|\sqrt{Z_h(q(\s))+2hm}\,d\s.
\end{equation}
From \eqref{eq-heur2} we conclude that 
\begin{equation} \label{eq-heur3}
\begin{split}
\int_0^1 \left(\min_{t'(\s)}f(t'(\s)) \right)\,d\s&=\mathcal{L}_h(q).
\end{split}
\end{equation}
Relation \eqref{eq-heur3} is reminiscent of the well-known connection between action and length in the non-relativistic case, suggesting that solutions of equation \eqref{2.1} can be obtained as geodesics with respect to a suitable Riemannian metric. 

\subsection{The formal statement: energy and Maupertuis functionals}

In this section we provide formal statements and proofs of the variational principle. 
Since, as well known, the length functional \eqref{eq-deflungh} has the drawback of being invariant under time-reparameterizations and is thus not suited for the calculus of variations and min-max critical point theory, 
we first follow the usual strategy of replacing the length functional with the energy functional: the resulting variational principle is contained in Proposition \ref{pr-var}. As a second step, we give an equivalent formulation of this principle, Proposition \ref{pr-maup}: in this second version, the energy functional is replaced by the so-called Maupertuis functional, an even more convenient choice for the use of critical point theory.

\smallbreak
Let us first fix $h\in \R$ such that $\Omega_h\neq \emptyset$ and consider the set
\[
\Lambda'_h=\{q\in H^1([0,1]): \ q(0)=q(1),\ q(\s)\in \Omega_h,\ \forall \ \s\in [0,1]\}.
\]
The \emph{energy functional} for problem \eqref{2.1} is defined by
\[
\mathcal{E}_h(q)=\int_0^1|q'(\s)|^2 \left(Z_h(q(\s))+2hm\right)\,d\s
\]
for every $q\in \Lambda'_h$.
It is well known (see the proof of Lemma \ref{pr-maup2} for more details) that critical points $q\in \Lambda'$ of $\mathcal{E}_h$ fulfill the geodesic equation
\begin{equation} \label{EL-eqn}
\frac{\rm d}{{\rm d}\s}(q'\, (Z_h(q)+2hm))=\frac{1}{2}|q'|^2\nabla Z_h(q)
\end{equation}
and, up to a reparameterization, are solutions of the second order equation
\[
z''=\nabla Z_h(z)
\]
satisfying the conservation law
\[
\dfrac{1}{2} |z'|^2-Z_h(z)=2hm.
\]
Nevertheless, up to a different reparameterization, they also give rise to solutions of fixed \textit{relativistic} energy of \eqref{2.1}. This is the main result of this section and is the content of the next Proposition.
\begin{proposition}
	\label{pr-var} 
	The following hold true.
	\smallbreak
	\noindent
	\textbf{Part A.} Let $q\in \Lambda'_h$ be a non-constant critical point of $\mathcal{E}_h$ and define 
	\begin{equation} \label{eq-cambvar1}
	t(\s)=\int_0^\s \dfrac{V(q(\xi))+h+mc^2}{\sqrt{Z_h(q(\xi))+2hm}}\, \dfrac{|q'(\xi)|}{c^2}\, d\xi,
	\end{equation}
	for every $\s \in [0,1]$. Let $T=t(1)$ and $\varsigma:[0,T]\to [0,1]$ be the inverse of $t$; define
	\[
	x(t)=q(\varsigma(t)), 
	\]
	for every $t \in [0,T]$. Then:
\begin{enumerate}
	\item $x(t)\in \Omega_h$ and $|\dot{x}(t)|< c$, for every $t\in [0,T]$,
	\item $x$ is a periodic solution of \eqref{2.1},
	\item $x$ satisfies 
	\begin{equation}
	\label{2.8bis}
	\begin{array}{ll}
	&\displaystyle\frac{mc^2}{\sqrt{1-|\dot{x}(t)|^2/c^2}}-V(x(t))=h+mc^2,\quad \forall \ t\in [0,T].
	\end{array}
	\end{equation}
\end{enumerate}

\smallskip
\noindent
\noindent
\textbf{Part B.} Let $x: [0,T] \to \Omega_h$ be a non-constant periodic solution of \eqref{2.1}, satisfying the energy law \eqref{2.8bis}.
Define
$$
\varpi(t) = \frac{c^2}{\sqrt{2\lambda}} \int_0^t \frac{Z_h(x(\xi))+2hm}{V(x(\xi)) + h + mc^2}\,d\xi,
$$
for every $t \in [0,T]$, where $\lambda$ is such that $\varpi(T) = 1$. Let $\tau: [0,1] \to [0,T]$ be the inverse of $\varpi$ and define 
$$
q(\s) = x(\tau(\s)), 
$$
for every $\s \in [0,1]$.
Then $q \in \Lambda_h'$ is a non-constant critical point of $\mathcal{E}_h$.
\end{proposition}
\begin{proof} \textbf{Part A.} 
	
	\noindent
1. From the definition of $x$ and the fact that $q\in \Lambda'_h$ we immediately deduce that $x(t)\in \Omega_h$, for every $t\in [0,T]$. Moreover, from \eqref{eq-cambvar1} we deduce that
	\begin{equation} \label{eq-derxq}
	\dot{x}(t)=q'(\varsigma(t))\, \dot{\varsigma}(t)=c^2\, \dfrac{q'(\varsigma(t))}{|q'(\varsigma(t))|}\, \dfrac{\sqrt{Z_h(q(\varsigma(t)))+2hm}}{V(q(\varsigma(t)))+h+mc^2},
	\end{equation}
	for every $t\in [0,T]$. Recalling \eqref{def-Z}, we have
	\[
	\begin{split}
	\dfrac{\sqrt{Z_h(q(\varsigma(t)))+2hm}}{V(q(\varsigma(t)))+h+mc^2}&=\dfrac{\sqrt{2mV(q(\varsigma(t)))+\dfrac{1}{c^2}\, (V(q(\varsigma(t)))+h)^2+2hm}}{V(q(\varsigma(t)))+h+mc^2}\\
	& =\dfrac{1}{c} \dfrac{\sqrt{(V(q(\varsigma(t)))+h)\,  (V(q(\varsigma(t)))+h+2mc^2)}}{V(q(\varsigma(t)))+h+mc^2}<\dfrac{1}{c},
	\end{split}
	\]
	for every $t\in [0,T]$. As a consequence, from \eqref{eq-derxq} we infer
	\[
	|\dot{x}(t)|<c,\quad \forall \ t\in [0,T].
	\]
\noindent
2. Let us now prove that $x$ is a periodic solution of \eqref{2.1}. To this aim, let us first observe that the periodicity of $x$ immediately follows from the periodicity of $q$. 

\noindent
As already observed, if $q\in \Lambda'_h$ is a critical point of $\mathcal{E}_h$, then $q$ is a solution of the geodesic equation \eqref{EL-eqn}. Moreover, a computation shows that $\frac{1}{2}|q'|^2(Z_h(q)+2hm)$ is a first integral of \eqref{EL-eqn}. Then, there exists $\lambda>0$ such that
\begin{equation}
\label{2.7}
\frac{1}{2}|q'(\s)|^2 (Z_h(q(\s))+2hm)=\l,\quad \forall \ \sigma\in[0,1].
\end{equation}
As a consequence, from \eqref{eq-cambvar1} we obtain
\begin{equation}
\label{repar}
\begin{split}
t'(\s)&=\dfrac{\sqrt{2\l}}{c^2}\dfrac{V(q(\s))+h+mc^2}{Z_h(q(\s))+2hm},
\end{split}
\end{equation}
for every $\s \in [0,1]$. On the other hand, from \eqref{def-Z} and \eqref{eq-derxq} we deduce that
\begin{equation}
\label{energy}
\begin{array}{l}
\displaystyle \sqrt{1-\dfrac{|\dot{x}(t)|^2}{c^2}}=\sqrt{1-c^2 \dfrac{Z_h(q(\varsigma(t)))+2hm}{(V(q(\varsigma(t)))+h+mc^2)^2}}
=\frac{mc^2}{V(x(t))+h+mc^2},
\end{array}
\end{equation}
for every $t\in [0,T]$. Hence, we obtain
\begin{equation}
\label{2.11}
\begin{split}
\frac{m\dot{x}(t)}{\sqrt{1-|\dot{x}(t)|^2/c^2}}& =\dfrac{V(q(\varsigma(t)))+h+mc^2}{c^2}\, q'(\varsigma(t))\, \dot{\varsigma}(t)\\
&
=\dfrac{V(q(\varsigma(t)))+h+mc^2}{c^2}\, \dfrac{c^2}{\sqrt{2\l}}\, \dfrac{Z_h(q(\varsigma(t)))+2hm}{V(q(\varsigma(t)))+h+mc^2}\, q'(\varsigma(t))\\
& =
\dfrac{1}{\sqrt{2\l}}\, q'(\varsigma(t))\, (Z_h(q(\varsigma(t)))+2hm), 
\end{split}
\end{equation}
for every $t\in [0,T]$. Hence, from \eqref{EL-eqn}, \eqref{repar} and \eqref{2.11} it follows that
\begin{equation}
\label{2.12}
\begin{split}
\frac{\rm d}{{\rm d}t}\left(\frac{m\dot{x}(t)}{\sqrt{1-|\dot{x}(t)|^2/c^2}}\right)&=\frac{1}{\sqrt{2\l}}\left(\frac{{\rm d}}{{\rm d} \sigma}\left(q'(\s)\, (Z_h(q(\s)+2hm) \right)\right)\vert_{\s=\varsigma(t)}\, \dot{\varsigma}(t)\\&=\dfrac{1}{\sqrt{2\l}}\frac{1}{2}|q'(\varsigma(t))|^2\nabla Z_h(q(\varsigma(t)))\, \dot{\varsigma}(t)\\&=\dfrac{c^2}{4\l}\, \dfrac{Z_h(q(\varsigma(t)))+2hm}{V(q(\varsigma(t)))+h+mc^2}\, |q'(\varsigma(t))|^2\, \nabla Z_h(q(\varsigma(t)),
\end{split}
\end{equation}
for every $t\in [0,T]$. Now, let us observe that \eqref{def-Z} implies that
\[
\nabla Z_h(w)=\dfrac{2}{c^2}\, \left(V(w)+h+mc^2\right)\, \nabla V(w),\quad \forall \ w\in \Omega_h.
\]
Taking into account \eqref{2.7} and this relation, from \eqref{2.12} we deduce that 
\[
\frac{\rm d}{{\rm d}t}\left(\frac{m\dot{x}(t)}{\sqrt{1-|\dot{x}(t)|^2/c^2}}\right)=\frac{c^2\nabla Z_h(q(\varsigma (t)))}{2(V(q(\varsigma (t)))+h+mc^2)}\\
=\nabla V(q(\varsigma (t)))=\nabla V(x(t)),
\]
for every $t\in [0,T]$.

\noindent
3. The energy law \eqref{2.8bis} immediately follows from \eqref{energy}.

\smallskip
\noindent
\noindent
\textbf{Part B.} Let us first observe that the fact that $q\in \Lambda'_h$ plainly follows from the definition of $q$. We now check that $q$ satisfies the geodesic equation \eqref{EL-eqn}: for every $\s \in [0,1]$ we have
\begin{equation} \label{eq-pr55}
q'(\s)=\dot{x}(t(\s))\, t'(\s)=\dot{x}(t(\s))\, \dfrac{\sqrt{2\lambda}}{c^2}\, \dfrac{V(x(t(\s)))+h+mc^2}{Z_h(x(t(\s)))+2hm}
\end{equation}
Moreover, from the energy law \eqref{2.8bis} and the definition of $Z_h$ given in \eqref{def-Z}, we deduce that
\[
|\dot{x}(t)|^2=c^4\, \dfrac{Z_h(x(t))+2hm}{(V(x(t))+h+mc^2)^2},
\]
for every $t\in [0,T]$. As a consequence, we obtain
\begin{equation} \label{eq-pr56}
\dfrac{1}{2}\, |q'(\s)|^2\, (Z_h(q(\s))+2hm)=\lambda, 
\end{equation}
for every $\s \in [0,1]$. On the other hand, recalling again the energy law \eqref{2.8bis}, from \eqref{eq-pr55} we deduce that
\[
q'(\s)=\dot{x}(t(\s))\, t'(\s)=\dfrac{\sqrt{2\lambda}}{c^2}\, \dfrac{1}{Z_h(x(t(\s)))+2hm}\, \dfrac{mc^2 \dot{x}(t(\s))}{\sqrt{1-|\dot{x}(t(\s))|^2/c^2}}
\]
and
\[
q'(\s)\, \left(Z_h(q(\s))+2hm\right) = \sqrt{2\lambda}\, \dfrac{m \dot{x}(t(\s))}{\sqrt{1-|\dot{x}(t(\s))|^2/c^2}},
\]
for every $\s \in [0,1]$. Now, from the fact that $x$ is a solution of \eqref{2.1} we infer that
\[
\begin{array}{ll}
\displaystyle \frac{\rm d}{{\rm d}\s}(q'(\s)\, (Z_h(q(\s))+2hm))&=\displaystyle \sqrt{2\lambda}\, \nabla V(x(t(\s)))\, t'(\s)\\ 
&\\
&\displaystyle =
 \dfrac{2\lambda}{c^2}\, \dfrac{V(q(\s))+h+mc^2}{Z_h(q(\s))+2hm}\, \nabla V(q(\s)),
 \end{array}
\]
for every $\s \in [0,1]$. Recalling \eqref{eq-pr56} and the fact that
\[
\nabla Z_h(x)=\dfrac{2}{c^2}\, (V(x)+h+mc^2)\, \nabla V(x),\quad \forall \ x\in \Omega,
\]
we conclude that $q$ satisfies \eqref{EL-eqn}.
\end{proof}

Let us now introduce the \emph{Maupertuis functional} defined by
\begin{equation}
\label{mau-func}
\mathcal{M}_h(u)=\int_0^1|u'(s)|^2\,ds\, \int_0^1 \left(Z_h(u(s))+2hm\right)\,ds,
\end{equation}
for every $u\in \Lambda'_h$; we point out that this is the functional that will be used in our application of Section \ref{sec3}.
The relation between critical points of $\mathcal{E}_h$ and $\mathcal{M}_h$ is well known: we recall it in the next lemma, whose proof is given for the reader's convenience at the end of this section.

\begin{lemma}\label{pr-maup2}
Let $q \in \Lambda'_h$ be a non-constant critical point of $\mathcal{E}_h$. Then, there exists a change of variable $\chi: [0,1] \to [0,1]$ such that the function
$$
u(s) = q(\chi(s)), \qquad s \in [0,1],
$$
is a non-constant critical point of $\mathcal{M}_h$. Conversely, if $u \in \Lambda'_h$ is a non-constant critical point of $\mathcal{M}_h$, then there exists a change of variable $\varrho: [0,1] \to [0,1]$ such that the function
$$
q(\s) = u(\varrho(\s)), \qquad \s\in [0,1],
$$
is a non-constant critical point of $\mathcal{E}_h$.
\end{lemma}

From Proposition \ref{pr-var} and Lemma \ref{pr-maup2} we can thus  obtain the following equivalent formulation of the variational principle, relating critical points of \eqref{mau-func} and solutions with prescribed relativistic energy of \eqref{2.1}.

\begin{proposition}
	\label{pr-maup}
	Let $u\in \Lambda'_h$ be a non-constant critical point of $\mathcal{M}_h$. Then, there exist $T>0$ and a change of variable $\upsilon:[0,T]\to [0,1]$ such that the function $x:[0,T]\to \R^n$ defined by
	\[
	x(t)=u(\upsilon(t)), \qquad t \in [0,T],
	\]
	is a non-constant periodic solution of \eqref{2.1}, satisfying the energy law \eqref{2.8bis}.
	Conversely, let $x: [0,T] \to \Omega_h$ be a non-constant periodic solution of \eqref{2.1}, satisfying the energy law \eqref{2.8bis}.
	Then, there exists a change of variable $\varphi:[0,1]\to [0,T]$ such that the function $u:[0,1]\to \R^n$ defined by
	\[
	u(s)=x(\varphi(s)), \qquad s \in [0,1],
	\]
	is a non-constant critical point of $\mathcal{M}_h$ in $\Lambda'_h$.
\end{proposition}

\begin{proof}[Proof of Lemma \ref{pr-maup2}] In order to prove the result, it is sufficient to show that non-constant critical points of the two functionals $\mathcal{E}_h$ and $\mathcal{M}_h$ correspond, up to suitable reparameterizations, to solutions $z$ of the differential equation 
	\begin{equation} \label{eq-secondordine}
	z''=\nabla Z_h(z)
	\end{equation}
satisfying
\begin{equation} \label{eq-energiavecchia}
\dfrac{1}{2}\, |z'|^2 -Z_h(z)=2hm.
\end{equation}
In the case of the Maupertuis functional, this follows from a well-known elementary computation: indeed, non-constant critical points $u\in \Lambda'_h$ of $\mathcal{M}_h$ and solutions $z$ of \eqref{eq-secondordine}-\eqref{eq-energiavecchia} are related by 
\[
z(\rho)=u\left(\dfrac{\rho}{\omega}\right),\quad \forall \ \rho \in [0,\omega],  
\]
where $\omega>0$ is given by
\[
\displaystyle \omega=\left(\dfrac{\bigintss_0^1 |u'(\s)|^2\,d\s}{2\bigintss_0^1 \left(Z_h(u(\s))+2hm\right)\,d\s}\right)^{\frac{1}{2}}.
\]
The relation bewteen non-constant critical points $q\in \Lambda'_h$ of $\mathcal{E}_h$ and solutions $z$ of \eqref{eq-secondordine}-\eqref{eq-energiavecchia} is less obvious. We write here the expressions of the change of variables relating $q$ and $z$ and we refer to \cite{Ben83} for the details of the proof.

\noindent
Given $q\in \Lambda'_h$, define
\[
\varrho(s)=\dfrac{1}{\sqrt{2}}\, \int_0^s \dfrac{|q'(\xi)|}{\sqrt{Z_h(q(\xi))+2hm}},\quad \forall \ s\in [0,1],
\]
and let $\varsigma$ be the inverse of $\varrho$. Then, the function $z:[0,\varrho(1)]\to \R^n$ defined by
\[
z(\rho)=q(\varsigma(\rho)),\quad \forall \ \rho \in [0,\varrho(1)],
\]
satisfies \eqref{eq-secondordine} and \eqref{eq-energiavecchia}.

\noindent
Conversely, for every function $z$ satisfying \eqref{eq-secondordine} and \eqref{eq-energiavecchia} in some interval $[0,\textrm{T}]$, let us define
\[
s (t )=\dfrac{1}{\lambda}\, \int_0^t (Z_h(z(\xi))+2hm)\, d\xi,\quad \forall \ t \in [0,\textrm{T}],
\]
where $\lambda >0$ is taken in such a way that $s(\textrm{T})=1$, and let $\tau$ be the inverse of $s$. Then, the function $q:[0,1]\to \R^n$ defined by
\[
q(\sigma)=z(\tau(s)),\quad \forall \ \sigma \in [0,1],
\]
is a non-constant critical point of $\mathcal{E}_h$.
\end{proof}

\section{Prescribed energy periodic solutions for a relativistic problem with singular potential}\label{sec3}

In this section, we present an application of the Maupertuis principle developed in Section \ref{sec2} to the existence of fixed energy periodic solutions for a relativistic problem with singular potential (for the case of smooth potentials, see Remark \ref{remsmooth}).

More precisely, we deal with the equation
\begin{equation}
\label{eq-ncariche}
\frac{\rm d}{{\rm d}t}\left(\frac{m\dot{x}}{\sqrt{1-|\dot{x}|^2/c^2}}\right)=- \sum_{i=1}^N \frac{\kappa_i (x-\sigma_i)}{\vert x - \sigma_i \vert^{\alpha+2}} + \nabla W(x), \qquad x \in \mathbb{R}^2 \setminus \{\sigma_1,\ldots,\sigma_N\},
\end{equation}
where, as usual, $m,c > 0$, $\sigma_1,\ldots,\sigma_N$ are distinct points of $\mathbb{R}^2$, $\kappa_i > 0$ for $i=1,\ldots,N$, and $W: \mathbb{R}^2 \to \mathbb{R}$ is a function of class $C^1$ satisfying the following assumptions: 
\begin{itemize}
\item[(W1)] there exists $M > 0$ such that $0 < W(x) \leq M$ for every $x \in \mathbb{R}^2$
\item[(W2)] $\nabla W(x) \to 0$ for $\vert x \vert \to +\infty$.
\end{itemize}
Finally, we assume the ``relativistic strong force condition'' 
\[
\alpha > 1,
\]
see Remark \ref{remstrong} for some comments. 
Of course, equation \eqref{eq-ncariche} is of the type \eqref{2.1} with 
\begin{equation}\label{potential}
V(x) = \sum_{i=1}^N \frac{\kappa_i}{\alpha \vert x - \sigma_i \vert^\alpha } + W(x)
\end{equation}
and $\Omega = \mathbb{R}^2 \setminus \{ \sigma_1,\ldots,\sigma_N\}$, so that the theory developed in Section \ref{sec2} can be applied.

In the above setting, we are going to prove the existence of (non-collision) periodic solutions of equation \eqref{eq-ncariche} with prescribed energy (in the sense of \eqref{defenergy}) and prescribed homotopy class in the punctured plane $\mathbb{R}^2 \setminus \{ \sigma_1,\ldots,\sigma_N\}$ (see Remark \ref{remhomotopy}). More precisely, we establish the following result.

\begin{theorem}\label{thmain}
In the above setting, let us fix $h > 0$ and a non-trivial homotopy class $[\gamma] \in \pi_1 (\mathbb{R}^2 \setminus  \{ \sigma_1,\ldots,\sigma_N\})$. Then, there exists a non-collision periodic solution of equation \eqref{eq-ncariche}, having energy $\mathfrak{E} = h + mc^2$ and belonging to the homotopy class $[\gamma]$.
\end{theorem}

\begin{remark}\label{remhomotopy}
\rm
Since we deal with solutions with prescribed energy (and unprescribed period), the expression ``belonging to the homotopy class $[\gamma]$'' used in the statement of Theorem \ref{thmain} requires some clarification. In fact, the periodic solution $x$ provided by Theorem \ref{thmain} arises as a reparameterization of a critical point $u: [0,1] \to \mathbb{R}^2 \setminus \{ \sigma_1,\ldots,\sigma_N\}$ of the Maupertuis functional. The condition that we impose is that the path $u$ belongs to the homotopy class 
$[\gamma]$; in terms of the solution $x$, this just means that, on some period $T > 0$ (unknown a-priori, and not necessarily the minimal one ), the path $x|_{[0,T]}$ belong (after a linear reparameterization of time) to the class $[\gamma]$. As a consequence, it is not possible to ensure that all the solutions provided by Theorem \ref{thmain}, in spite of coming from different homotopy classes, are distinct the one from the other.

To understand this delicate issue, it can be useful to first focus on the case $N=1$. In this situation, as well known, the fundamental group $\pi_1 (\mathbb{R}^2 \setminus \{\sigma_1\})$ is isomorphic to $\mathbb{Z}$ and the homotopy class of a closed path is completely determined by its winding number around the point $\sigma_1$. Thus, in principle we could apply Theorem \ref{thmain} to get the existence of a periodic solution $x_k$ with winding number $k$, for any $k \in \mathbb{Z}$, with $k \neq 0$.
If $k = \pm 1$, of course we can be sure that the winding number is evaluated on the minimal period, say $T_k$, of the solution $x_k$. 
On the other hand, a solution $x_k$ having winding number $k$ with $\vert k \vert > 1$ could, in principle, 
just correspond to the solution $x_{\textnormal{sgn}(k)}$, when regarded on the period $\vert k \vert T_k$.
Thus, summing up, in the case $N = 1$ Theorem \ref{thmain} just ensures the existence of two solutions, a first one winding around the origin in the clockwise sense and the second one in the counter-clockwise sense. Since, moreover, the problem is time-reversible, these two solutions could correspond to the same orbit, just covered in different directions: summing, just \emph{one} periodic solution can be provided by Theorem \ref{thmain} in this case.

In the case $N \geq 2$, the situation is similar, but more involved. Indeed, now the fundamental group $\pi_1 (\mathbb{R}^2 \setminus \{\sigma_1, \ldots, \sigma_N\})$ is homeomorphic to the free group with $n$ generators $\{a_1,\ldots,a_n\}$, that is, an homotopy class is described by a finite word formed with the symbols $a_i$ and their inverses $a_i^{-1}$ (cf. \cite{Cas17,Hat02}). Taking into account that the group product corresponds to the concatenation of paths, we deduce that homotopy classes described by words which are formed by a repetition of a same sub-word (like, for instance, $a_1 a_2^{-1} a_1 a_2^{-1}$, which consists of two repetitions of the sub-word $a_1 a_2^{-1}$) must not be considered, since they could give rise to the same periodic solutions provided by the homotopy class described by the sub-word.
Moreover, again by the time-reversibility of the problem, a word and its inverse must be identified, since they could provide the same periodic solution, up to time-inversion. Clearly enough, however, still an infinite number of distinct periodic solutions (with the same energy) can be provided by Theorem \ref{thmain}  for any $N \geq 2$.
\end{remark}

We now give the proof of Theorem \ref{thmain}. 
Before going into the details, however, let us state and prove the following preliminary lemma.

\begin{lemma}\label{poincare}
Let $u \in H^1([0,1])$ be such that $u(0) = u(1)$ and $u(s) \neq \sigma_i$ for every $s \in [0,1]$ and $i=1,\ldots,N$.
If, moreover, $u$ is not null-homotopic as a loop in $\mathbb{R}^2 \setminus \{\sigma_1,\ldots,\sigma_N\}$, then
\begin{equation}\label{poinc}
\Vert u \Vert_{L^\infty} \leq R + \Vert u' \Vert_{L^2},
\end{equation}
where $R = \max_i \vert \sigma_i \vert$.
\end{lemma}

\begin{proof}
Let us assume that $\Vert u \Vert_{L^\infty} > R$ (otherwise there is nothing to prove) and take $s^* \in [0,1]$ such that
$\vert u(s^*) \vert = \Vert u \Vert_{L^\infty}$. Now, let us consider the open ball $B \subset \mathbb{R}^2$ centered at $u(s^*)$ and with radius $r=\vert u(s^*) \vert - R$: clearly, such a ball does not contain any of the centers $\{\sigma_1,\ldots,\sigma_N\}$. Accordingly, the image of the path $u$ cannot be entirely contained in $B$, since otherwise $u$ would be a null-homotopic loop in $\mathbb{R}^2 \setminus \{\sigma_1,\ldots,\sigma_N\}$. We thus infer that there exists $\bar{s} \in [0,1]$ such that
$$
\vert u(\bar{s}) - u(s^*) \vert \geq \vert u(s^*) \vert - R.
$$
Then
\begin{align*}
\Vert u \Vert_{L^\infty} & = \vert u(s^*) \vert \leq R + \vert u(\bar{s}) - u(s^*) \vert \leq R + \left\vert \int_{s^*}^{\bar{s}} \vert u' \vert \,ds \right\vert \\
& \leq R + \vert s^* - \bar{s} \vert^{1/2} \left\vert\int_{s^*}^{\bar{s}} \vert u' \vert^2 \,ds \right\vert^{1/2} \leq R + \Vert u' \Vert_{L^2},
\end{align*}
thus proving \eqref{poinc}.
\end{proof}

\begin{proof}[Proof of Theorem \ref{thmain}]
Let us first observe that, since $V(u) > 0$ for every $u \in \mathbb{R}^2 \setminus \{\sigma_1,\ldots,\sigma_N\}$ and $h > 0$, with the notation of Section \ref{sec2} we have $\Omega_h = \Omega = \mathbb{R}^2 \setminus \{ \sigma_1,\ldots,\sigma_N\}$ and
$$
\Lambda_h' = \{ u \in H^1([0,1]) \, : \, u(0) = u(1), \; u(s) \neq \sigma_i\; \forall \ s \in [0,1], \, i=1,\ldots,N\}.
$$
Accordingly, we can consider the Maupertuis functional $\mathcal{M}_h$ as defined in \eqref{mau-func}, that is,
\begin{equation}
\label{mau-func2}
\mathcal{M}_h(u)=\int_0^1|u'|^2\,ds\, \int_0^1 \left(Z_h(u)+2hm\right)\,ds, \qquad u \in \Lambda_h',
\end{equation}
where, as in \eqref{def-Z}, 
\begin{equation} \label{def-Z2}
Z_h(u)=2mV(u)+\dfrac{1}{c^2}\, (V(u)+h)^2,
\end{equation}
for $u \in \Omega$. By standard arguments (cf. \cite[Theorem 1.4]{MW89}), the functional $\mathcal{M}_h$ is of class $C^1$, with
\begin{equation}\label{Mc1}
d\mathcal{M}_h(u)[v] = 2 \int_0^1 \langle u', v' \rangle \,ds \int_0^1 \left(Z_h(u)+2hm\right)\,ds + 
\int_0^1 |u'|^2\,ds \int_0^1 \langle \nabla Z_h(u),v \rangle \,ds.
\end{equation}

For further convenience, we now claim that $Z_h$ satisfies the following crucial properties: for every $i=1,\ldots,N$, there exist $\beta_i > 0$ and $r_i > 0$ such that
\begin{equation}\label{p-Z1}
Z_h(u) \geq \frac{\beta_i}{\vert u -\sigma_i \vert^2}, \quad \mbox{ if } 0 < \vert u - \sigma_i \vert < r_i,
\end{equation}
and, moreover,
\begin{equation}\label{p-Z2}
\lim_{u \to \sigma_i} \left( Z_h(u) + \frac{1}{2} \langle \nabla Z_h(u),u-\sigma_i \rangle \right) = - \infty.
\end{equation}
Indeed, \eqref{p-Z1} is an easy consequence of the fact that $\alpha > 1$. As far as \eqref{p-Z2} is concerned, recalling \eqref{def-Z2}, we immediately deduce that
\begin{equation} \label{eq-pr101}
\begin{array}{ll}
\displaystyle Z_h(u) + \frac{1}{2} \langle \nabla Z_h(u),u-\sigma_i \rangle = & \displaystyle \dfrac{1}{c^2}\, (V(u))^2+\left(2m+\dfrac{2h}{c^2}\right)\, V(u)+\dfrac{h^2}{c^2}\\
&\\
&\displaystyle + \left(\dfrac{V(u)}{c^2}+m+\dfrac{h}{c^2}\right)\, \langle \nabla V(u),u-\sigma_i\rangle,
\end{array}
\end{equation}
for every $u\in \Omega$ and $i=1,\ldots,N$. Now, let us observe that \eqref{potential} implies that for every $i=1,\ldots,N$ there exists a neighbourhood $\Sigma_i$ of $\sigma_i$ such that
\[
V(u)=\dfrac{\kappa_i}{\alpha \ |u-\sigma_i|^\alpha}+ \widetilde{W}(u),\quad \forall \ u\in \Sigma_i,\ u\neq \sigma_i,
\]
with $\widetilde W$ and $\nabla \widetilde W$ bounded in $\Sigma_i$. Hence, from \eqref{eq-pr101} we infer that there exist a constant $K\in \R$ and a function $W^*$ bounded in $\Sigma_i$ such that
\[
Z_h(u) + \frac{1}{2} \langle \nabla Z_h(u),u-\sigma_i \rangle = \dfrac{\kappa_i^2}{\alpha^2 c^2}\, (1-\alpha)\, \dfrac{1}{|u-\sigma_i|^{2\alpha}}+\dfrac{K}{|u-\sigma_i|^\alpha}+W^*(u),
\]
for every $u\in \Sigma_i$, $u\neq \sigma_i$. From this relation, recalling that $\alpha>1$ we plainly obtain \eqref{p-Z2}.

Let us now notice that from \eqref{p-Z2} we immediately infer that there exists $\varepsilon \in (0,1)$ such that, for every $i =1,\ldots,N$, 
\begin{equation}\label{puntadentro}
d\mathcal{M}_h(u)[u-\sigma_i] < 0, \quad \mbox{ if } \Vert u - \sigma_i \Vert \leq 2 \varepsilon, \quad \mbox{$u$ non-constant}.
\end{equation}
Indeed, by \eqref{Mc1} we have
$$
d\mathcal{M}_h(u)[u-\sigma_i] = 2 \int_0^1 |u'|^2\,ds \int_0^1 \left( Z_h(u)+2hm + \frac{1}{2} \langle \nabla Z_h(u),u-\sigma_i \rangle\right)\,ds,
$$
so that the conclusion follows from \eqref{p-Z2}, in view of the continuous embedding $H^1([0,1]) \subset L^\infty ([0,1])$.
Of course, such a $\varepsilon$ can be chosen so small that
\begin{equation}\label{disgiunte}
\Vert u - \sigma_i \Vert \leq 2 \varepsilon \; \Longrightarrow \; \Vert u - \sigma_j \Vert > 2 \varepsilon,\, \forall \ j \neq i,
\end{equation}
that is, the balls $\{B_{2\varepsilon}(\sigma_i)\}_{i=1}^N$ are disjoint. From now on, we assume that this condition is satisfied.

We are now in a position to start with the proof. Fixed a non-trivial homotopy class $[\gamma] \in \pi_1 (\mathbb{R}^2 \setminus  \{ \sigma_1,\ldots,\sigma_N\})$, and with $\varepsilon > 0$ as above, let us consider the set
$$
\Lambda''_{h,\gamma,\varepsilon} = \left\{ u \in \Lambda'_h \, : \, u \mbox{ lies in the homotopy class $[\gamma]$ and } \Vert u - \sigma_i \Vert \geq \varepsilon, \, \forall \ i =1,\ldots,N \right\}.
$$
Observe that, by the non-triviality of $\gamma$, any $u \in \Lambda''_{h,\gamma,\varepsilon}$ is non-constant 
Moreover, recalling that $Z_h$ is always positive, we plainly deduce the estimate
\begin{equation}\label{coerciv}
\mathcal{M}_h(u) \geq 2hm \int_0^1  |u'|^2\,ds > 0, \quad \forall \ u \in \Lambda''_{h,\gamma,\varepsilon}.
\end{equation}
With this in mind, let us consider the minimization problem
$$
\inf_{u \in \Lambda''_{h,\gamma,\varepsilon}} \mathcal{M}_h.
$$
We split our next arguments into some steps.
\smallbreak
\noindent
\underline{Step 1}: \emph{convergence of minimizing sequences.} 
Let us consider an arbitrary minimizing sequence $\{u_n\} \subset \Lambda''_{h,\gamma,\varepsilon}$ for the minimization problem
$$
\inf_{u \in \Lambda''_{h,\gamma,\varepsilon}} \mathcal{M}_h.
$$
By the estimate \eqref{coerciv}, the sequence $\Vert \dot u_n \Vert_{L^2}$ is bounded; using \eqref{poinc}, it thus follows that 
$\Vert u_n \Vert_{L^\infty}$ is bounded as well. Since $\Vert u_n \Vert_{L^2} \leq \Vert u_n \Vert_{L^\infty}$, the sequence $\{u_n\}$ is thus bounded in the $H^1$-norm.
Hence, there exists $\bar{u} \in H^1([0,1])$, with $\bar{u}(0) = \bar{u}(1)$, such that (up to subsequences, which we do not relabel) $u_n$ converges to $\bar{u}$ weakly in $H^1$ and, by compactness, strongly in $L^\infty$.

We now claim that $\bar{u}$ does not have collisions, that is, $\bar{u}(s) \neq \sigma_i$ for every $s \in [0,1]$ and $i \in \{1,\ldots,N\}$
(and, hence, $\bar{u} \in \Lambda_h'$).
Indeed, let us assume by contradiction that this is not true. Then, since \eqref{p-Z1} holds true, the same argument used in \cite[Lemma 5.3]{ACZ93} implies that
$$
\int_0^1 Z_h(u_n) \,ds \to +\infty.
$$
Hence, a contradiction with the fact that $\{u_n\}$ is a minimizing sequence is obtained as long as we prove that 
$$
\liminf_{n \to +\infty} \int_0^1  |u'_n|^2\,ds > 0.
$$
To see that this is true, we argue again by contradiction, assuming that (up to subsequences) $\int_0^1 |\dot u_n|^2\,ds \to 0$.
Then, for any $s_1, s_2 \in [0,1]$, passing to the limit in the inequality
$$
\vert u_n(s_2) - u_n(s_1) \vert \leq \left\vert \int_{s_1}^{s_2} \vert u'_n \vert \,ds  \right\vert\leq \int_0^1 |u'_n|^2\,ds 
$$
we see that $\bar{u}$ is constant and, thus, $\bar{u} \equiv \sigma_{i^*}$ for some fixed center $\sigma_{i^*}$.
Moreover, since the derivative of a constant function is identically zero and $\Vert \dot u_n \Vert_{L^2} \to 0$, we obtain that the convergence of $u_n$ to $\sigma_{i^*}$ is strong in $H^1$, a contradiction with the fact that $\Vert u_n - \sigma_{i^*} \Vert \geq \varepsilon$.

Let us notice that, from the just proved fact that $\bar{u}$ does not have collisions, and since homotopy classes are open w.r.t. the $L^\infty$-topology, the $L^\infty$-convergence of $u_n$ to $\bar{u}$ implies that $\bar{u}$ lies in the homotopy class $[\gamma]$, as well. 
\smallbreak
\noindent
\underline{Step 2}: \emph{construction of a minimizing sequence $\{\widetilde u_n\}$ away from the boundary $\Vert \widetilde u_n - \sigma_i \Vert \geq \varepsilon$.} 
Precisely, in this step we show that there exists a minimizing sequence $\{\widetilde u_n\} \subset \Lambda''_{h,\gamma,\varepsilon}$ such that, for some $\lambda \in (1,2]$, the inequality
\begin{equation}\label{step2}
\Vert \widetilde u_n - \sigma_i \Vert \geq \lambda \varepsilon, \quad \forall \ i \in \{1,\ldots,N\}.
\end{equation}
is satisfied for $n$ large enough.

To prove this, let us consider an arbitrary minimizing sequence $\{ u_n\} \subset \Lambda''_{h,\gamma,\varepsilon}$.
By Step 1, we can assume that $u_n$ converges, weakly in $H^1$ and strongly in $L^\infty$, to a limit function $\bar{u}$ which lies in the homotopy class $[\gamma]$. 
Since homotopy classes are open w.r.t. the $L^\infty$-topology, we can 
fix $\rho > 0$ such that
\begin{equation}\label{classestabile}
\Vert u - \bar{u} \Vert_{L^\infty} \leq \rho \; \Longrightarrow \; u \mbox{ lies in the homotopy class $[\gamma]$}.
\end{equation}
Of course, with such a choice of $\rho$ there exists a corresponding integer $n_\rho \geq 0$ such that
\begin{equation}\label{rhomezzi}
\Vert u_n - \bar{u} \Vert_{L^\infty} \leq \frac{\rho}{2}, \quad \forall \ n \geq n_\rho.
\end{equation}
As a consequence, we deduce that
\[
\Vert u_n - \sigma_i \Vert_{L^\infty}\leq \Vert u_n - \bar{u} \Vert_{L^\infty}+\Vert \bar{u} - \sigma_i \Vert_{L^\infty} \leq \frac{\rho}{2}+\Vert \bar{u} - \sigma_i \Vert_{L^\infty},
\]
implying that 
\[
\sup_{n\geq n_\rho} \max_i \Vert u_n - \sigma_i \Vert_{L^\infty}<+\infty.
\]
Hence, we can set 
$$
M = \max \left\{\sup_{n\geq n_\rho} \max_i \Vert u_n - \sigma_i \Vert_{L^\infty},\frac{\rho}{2}\right\}
$$
and
$$
\lambda = 1 + \frac{\rho}{2M}.
$$
We are now ready to construct the sequence $\{\widetilde u_n\}$ (defined for $n \geq n_\rho$).

So, let us assume that, for some $n \geq n_\rho$, the set 
$$
I_n = \{i \in \{1,\ldots,N\} : \Vert u_n - \sigma_i \Vert < \lambda \varepsilon \}
$$
is non-empty (otherwise, we can simply take $\widetilde u_n = u_n$). Since $\lambda \leq 2$, by \eqref{disgiunte} we have that there exists exactly one $i(n) \in \{1,\ldots,N\}$ such that 
$$
\Vert u_n - \sigma_{i(n)} \Vert < \lambda\varepsilon \quad \mbox{ while } \quad \Vert u_n - \sigma_j \Vert > \lambda\varepsilon, \, \forall \ j \neq i(n).
$$
Let us set
$$
k_n = \frac{\lambda \varepsilon}{\Vert u_n - \sigma_{i(n)} \Vert};
$$
since $\Vert u_n - \sigma_{i(n)} \Vert \geq \varepsilon$ by the definition of $\Lambda''_{h,\gamma,\varepsilon}$, we have
\begin{equation}\label{kn}
0 < k_n - 1 < \lambda - 1 = \frac{\rho}{2M}.
\end{equation}
Finally, let us define
\begin{equation}\label{deftildeu}
\widetilde u_n = \sigma_{i(n)} + k_n (u_n - \sigma_{i(n)}).
\end{equation}
Of course
\begin{equation}\label{utile}
\Vert \widetilde u_n - \sigma_{i(n)} \Vert = \lambda \varepsilon 
\end{equation}
and so, since $\lambda \leq 2$ and using again by \eqref{disgiunte}, 
\begin{equation}\label{utile2}
\Vert \widetilde u_n - \sigma_{j} \Vert > 2 \varepsilon, \quad \forall \ j \neq i(n).
\end{equation}
Moreover, using \eqref{rhomezzi} and \eqref{kn}
\begin{align*}
\Vert \widetilde u_n - \bar{u} \Vert_{L^\infty} & = \Vert u_n - \bar{u} + (1-k_n)(\sigma_{i(n)} - u_n) \Vert_{L^\infty} \\
& \leq \Vert u_n - \bar{u}  \Vert_{L^\infty} + (k_n-1) \Vert u_n - \sigma_{i(n)} \Vert_{L^\infty} \leq \frac{\rho}{2} + \frac{\rho}{2M}M = \rho.
\end{align*}
By \eqref{classestabile}, we thus have that $\widetilde u_n$ has the homotopy class $[\gamma]$; recalling \eqref{utile} and \eqref{utile2}, we conclude that
$u_n \in \Lambda''_{h,\gamma,\varepsilon}$. 
It thus remains to show that $\{\widetilde u_n\}$ is still a minimizing sequence: to this end, we are going to prove that
\begin{equation}\label{livelloscende}
\mathcal{M}_h(\widetilde u_n) \leq \mathcal{M}_h(u_n).
\end{equation}
So, let us assume that, for some $n$, $\widetilde u_n$ is defined as in \eqref{deftildeu} (otherwise $\widetilde u_n = u_n$ and there is nothing to prove). Accordingly, we define, for every $t \in [1,k_n]$, $U_t = t(u_n-\sigma_{i(n)})+\sigma_{i(n)}$ in such a way that
$U_1 = u_n$ and $U_{k_n} = \widetilde u_n$. Then, we have
\begin{equation} \label{eq-nuovissima}
\frac{d}{dt} \mathcal{M}_h(U_t) = d\mathcal{M}_h(U_t)[u_n-\sigma_{i(n)}] = \frac{1}{t} d\mathcal{M}_h(U_t)[t(u_n-\sigma_{i(n)})], 
\end{equation}
with $t \in [1,k_n]$. From \eqref{utile} we deduce that
$$
\Vert U_t - \sigma_{i(n)} \Vert \leq k_n \Vert u_n-\sigma_{i(n)} \Vert = \Vert \widetilde u_n - \sigma_{i(n)} \Vert = \lambda \varepsilon \leq 2 \varepsilon,
$$
where the last inequality follows from the fact that $\lambda \leq 2$. Since
\[
d\mathcal{M}_h(U_t)\left[U_t - \sigma_{i(n)}\right]=d\mathcal{M}_h(U_t)\left[t(u_n - \sigma_{i(n)})\right],
\]
from \eqref{puntadentro} and \eqref{eq-nuovissima} we obtain
$$
\frac{d}{dt} \mathcal{M}_h(U_t) < 0, \quad \forall k \in [1,k_n].
$$
This implies that $\mathcal{M}_h(U_{k_n}) < \mathcal{M}_h(U_1)$ and thus \eqref{livelloscende}.
\smallbreak
\noindent
\underline{Step 3}: \emph{construction of a minimizing sequence $\{ \hat u_n\}$ away from the boundary $\Vert \hat u_n - \sigma_i \Vert \geq \varepsilon$ and such that $d \mathcal{M}_h(\hat u_n) \to 0$.} Precisely, in this step we show that there exists a minimizing sequence $\{\widehat u_n\} \subset \Lambda''_{h,\gamma,\varepsilon}$ such that, for some $\lambda^* > 1$, the inequality
\begin{equation}\label{step3}
\Vert \widehat u_n - \sigma_i \Vert \geq \lambda^* \varepsilon, \quad \forall \ i \in \{1,\ldots,N\}.
\end{equation}
is satisfied for $n$ large enough and $d \mathcal{M}_h (\widehat u_n) \to 0$.

To this end, we consider a minimizing sequence $\{ \widetilde u_n \} \subset \Lambda''_{h,\gamma,\varepsilon}$ constructed as in Step 2,
and thus satisfying \eqref{step2}. By Step 1, we can assume that $\{ \widetilde u_n \}$ converges, weakly in $H^1$ and strongly in $L^\infty$, to a limit function $u^*$ without collisions and lying in the homotopy class $[\gamma]$. Hence, we can fix $\rho^* > 0$ so small that
the inequality
\begin{equation}\label{eke1}
\vert \widetilde u_n(s) - \sigma_i \vert \geq \rho^*,  \quad \forall \ s \in [0,1], \, i \in \{1,\ldots,N\},
\end{equation}
is satisfied if $n$ is large enough. Moreover, up to some relabeling of the sequence we can also assume that
\begin{equation}\label{eke0}
\mathcal{M}_h (\widetilde u_n) \leq \inf_{u \in \Lambda''_{h,\gamma,\varepsilon}} \mathcal{M}_h (u) + \frac{1}{\sqrt{n}},
\end{equation}
for $n$ large enough.

With this in mind, let us consider the set
$$
\Lambda^* = \{ u \in \Lambda''_{h,\gamma,\varepsilon} \, : \,  \vert \widetilde u(s) - \sigma_i \vert \geq \rho^*/2,  \quad \forall \ s \in [0,1], \, i \in \{1,\ldots,N\} \}.
$$
In view of the continuous embedding $H^1 \subset L^\infty$ (and recalling the fact that homotopy classes are open with respect to the $L^\infty$-topology), $M$ is (strongly) closed in $H^1$; moreover, by \eqref{eke1}, the sequence $\{ \widetilde u_n \}$ belongs to $\Lambda^*$ for $n$ large enough (and it is thus a minimizing sequence also for $\inf_{\Lambda^*} \mathcal{M}_h$, since $\Lambda^* \subset \Lambda''_{h,\gamma,\varepsilon}$). Taking into account \eqref{eke0}, by Ekeland variational principle (as stated, for instance, \cite[Theorem 4.1 and Remark 4.1]{MW89}) we can find a further sequence 
$\{ \widehat u_n \} \subset \Lambda^*$ such that
\begin{equation}\label{eke2}
\mathcal{M}_h(\widehat u_n) \leq \mathcal{M}_h (\widetilde u_n), \qquad \Vert \widehat u_n - \widetilde u_n \Vert \leq \frac{1}{\sqrt{n}}
\end{equation}
and
\begin{equation}\label{eke3}
\mathcal{M}_h(w) - \mathcal{M}_h(\widehat u_n) \geq -\frac{1}{\sqrt{n}} \Vert w - \widehat u_n \Vert, \quad \forall \ w \in \Lambda^*.
\end{equation}
By the first condition \eqref{eke2}, we deduce that the sequence $\{ \widehat u_n \}$ is still minimizing. Moreover, taking into account \eqref{step2}, the second condition in \eqref{eke2} yields
\begin{equation}\label{eke4}
\Vert \widehat u_n - \sigma_i \Vert \geq \frac{1+\lambda}{2} \varepsilon, \quad \forall \ i \in \{1,\ldots,N\}.
\end{equation}
for $n$ large enough, and thus \eqref{step3} for $\lambda^* = (1+\lambda)/2>1$. Finally, again from the second condition \eqref{eke2} together with the fact that $\widetilde u_n \to u^*$ strongly in $L^\infty$, we infer that $\widehat u_n \to u^*$ strongly in $L^\infty$ as well:
from \eqref{eke1}, we thus get 
\begin{equation}\label{eke5}
\vert \widehat u_n(s) - \sigma_i \vert \geq \frac{3}{4}\rho^* > \frac{\rho^*}{2},  \quad \forall \ s \in [0,1], \, i \in \{1,\ldots,N\},
\end{equation}
for $n$ large enough.

Now, let us observe that property \eqref{eke4} and \eqref{eke5} together imply that the sequence $\{ \widehat u_n \}$ belongs (at least for $n$ large enough) to the interior of the set $M$. From this observation, taking into account \eqref{eke3}, it easily follows that 
\begin{equation}\label{ekefinale}
d\mathcal{M}_h(\widehat u_n) \to 0.
\end{equation}
Indeed, let us fix an arbitrary $v \in H^1([0,1])$ with $v(0) = v(1)$ and $\Vert v \Vert = 1$.
Then, $\widetilde u_n + t v \in \Lambda^*$ if $\vert t \vert$ is small enough, and thus, using \eqref{eke3},
$$
d\mathcal{M}_h(\widehat u_n)[v] = \lim_{t \to 0^+} \frac{\mathcal{M}_h(\widehat u_n + tv) - \mathcal{M}_h(\widehat u_n)}{t} \geq -\frac{1}{\sqrt{n}}
$$
and
$$
d\mathcal{M}_h(\widehat u_n)[v] = \lim_{t \to 0^-} \frac{\mathcal{M}_h(\widehat u_n + tv) - \mathcal{M}_h(\widehat u_n)}{-\vert t \vert} \leq \frac{1}{\sqrt{n}}.
$$
This implies 
$$
\Vert d\mathcal{M}_h (\widehat u_n)\Vert = \sup_{\Vert v \Vert = 1} \vert d\mathcal{M}_h(\widehat u_n)[v] \vert \leq \frac{1}{\sqrt{n}}
$$
and thus \eqref{ekefinale}, as desired. 
\smallbreak
\noindent
\underline{Step 4}: \emph{strong convergence of the sequence $\{\widehat u_n \}$ and conclusion.}
Let us consider the minimizing sequence $\{\widehat u_n\} \subset \Lambda^* \subset \Lambda''_{h,\gamma,\varepsilon}$ constructed in Step 3; as already observed, $\widehat u_n$ converges weakly in $H^1$ and strongly in $L^\infty$ to a limit function $u^*$ without collisions and lying in the homotopy class $[\gamma]$. We now claim that
\begin{equation}\label{step4}
\int_0^1 \vert (\widehat u_n)' \vert^2 \,ds \to \int_0^1 \vert (u^*)' \vert^2 \,ds.
\end{equation}
From this, the conclusion follows easily. Indeed, \eqref{step4} implies (by a well-known property of Hilbert spaces) that $\widetilde u_n \to u^*$ strongly in $H^1$. As a consequence, $\mathcal{M}_h(\widehat u_n) \to \mathcal{M}_h(u^*)$ and thus $u^*$ is a minimum point for $\mathcal{M}_h$ on $\Lambda''_{h,\gamma,\varepsilon}$. Moreover, again by the strong $H^1$-convergence together with \eqref{eke4}, it holds
$$
\Vert u^*- \sigma_i \Vert \geq \frac{1+\lambda}{2} \varepsilon > \varepsilon, \quad \forall \ i \in \{1,\ldots,N\},
$$
so that $u^*$ belongs to the interior of the set $\Lambda''_{h,\gamma,\varepsilon}$. Hence, $u^*$ is a critical point (in fact, a local minimum) for $\mathcal{M}_h$ and, of course, it is non-constant since the homotopy class $[\gamma]$ is non-trivial. Thus, the conclusion follows from Proposition \ref{pr-maup}.

So, let us prove \eqref{step4}. 
Since $\{\widehat u_n \}$ is bounded in $H^1$, by \eqref{ekefinale} we have
\begin{equation}\label{fine0}
d\mathcal{M}_h(\widehat u_n)[\widehat u_n-u^*] \to 0.
\end{equation} 
On the other hand, recalling formula \eqref{Mc1} we have
\begin{align*}
d\mathcal{M}_h(\widehat u_n)[\widehat u_n] & = 2 \int_0^1 \vert (\widehat u_n)' \vert^2 \,ds \int_0^1 \left(Z_h(\widehat u_n)+2hm\right)\,ds \\ & \qquad \qquad + \int_0^1 |(\widehat u_n)'|^2\,ds \int_0^1 \langle \nabla Z_h(\widehat u_n),\widehat u_n \rangle \,ds \\
& = 2 \mathcal{M}_h(\widehat u_n) + \int_0^1 |(\widehat u_n)'|^2\,ds \int_0^1 \langle \nabla Z_h(\widehat u_n),\widehat u_n \rangle \,ds
\end{align*}
and
\begin{align*}
d\mathcal{M}_h(\widehat u_n)[u^*] & = 2 \int_0^1 \langle (\widehat u_n)',(u^*)' \rangle  \,ds \int_0^1 \left(Z_h(\widehat u_n)+2hm\right)\,ds \\ & \qquad \qquad + \int_0^1 |(\widehat u_n)'|^2\,ds \int_0^1 \langle \nabla Z_h(\widehat u_n),u^* \rangle \,ds,
\end{align*}
so that 
\begin{equation}\label{fine3}
\begin{split}
d\mathcal{M}_h(\widehat u_n)[\widehat u_n-u^*] &= 2 \mathcal{M}_h(\widehat u_n) + \int_0^1 |(\widehat u_n)'|^2\,ds \int_0^1 \langle \nabla Z_h(\widehat u_n),\widehat u_n - u^* \rangle \,ds \\
& \qquad \qquad - 2 \int_0^1 \langle  (\widehat u_n)',(u^*)'  \rangle  \,ds \int_0^1 \left(Z_h(\widehat u_n)+2hm\right)\,ds.
\end{split}
\end{equation}
Now, by the facts that $\{\widehat u_n\}$ is bounded in $H^1$ and converges  to $u^*$ strongly in $L^\infty$, we have
$$
\int_0^1 |(\widehat u_n)'|^2\,ds \int_0^1 \langle \nabla Z_h(\widehat u_n),\widehat u_n - u^* \rangle \,ds \to 0,
$$
while from the weak convergence in $H^1$ we obtain 
\begin{align*}
2 \int_0^1\langle  (\widehat u_n)',(u^*)'  \rangle  \,ds \int_0^1 \left(Z_h(\widehat u_n)+2hm\right)\,ds \rightarrow & \ 2 \int_0^1 \vert (u^*)' \vert^2 \,ds \int_0^1 \left(Z_h(u^*)+2hm\right)\,ds \\ & = \mathcal{M}_h(u^*).
\end{align*}
Therefore, \eqref{fine0} and \eqref{fine3} yields
$$
\mathcal{M}_h(\widehat u_n) \to \mathcal{M}_h(u^*)
$$
and thus, again by the convergence $\widehat u_n \to u^*$ in $L^\infty$,
\begin{align*}
\int_0^1 \vert (\widehat u_n)' \vert^2 \,ds &= \mathcal{M}_h(\widehat u_n) \left(\int_0^1 \left(Z_h(\widehat u_n)+2hm\right)\,ds\right)^{-1} \\
& \qquad \rightarrow \  \mathcal{M}_h(u^*) \left(\int_0^1 \left(Z_h(u^*)+2hm\right)\,ds\right)^{-1} = \int_0^1 \vert (u^*)' \vert^2 \,ds
\end{align*}
finally proving \eqref{step4}.
\end{proof}

\begin{remark}\label{rem-nonrel1}
	\rm
Let us recall, as already observed at the beginning of Section \ref{sec2}, that the difference $h$ between the relativistic energy $\mathfrak{E}$ and the rest energy converges to the classical mechanical energy in the non-relativistic limit (i.e. when $c\to +\infty$). Indeed, a simple computation shows that 
	\[ 
	\lim_{c\to +\infty} \left(\dfrac{mc^2}{\sqrt{1-|y|^2/c^2}}-mc^2\right)=\dfrac{1}{2}\, m|y|^2,\quad \forall \ y\in \R^n.
	\]
From this point of view, the classical system 
\begin{equation}\label{eq-class}
m \ddot x = \nabla V(x)
\end{equation} 
can formally be seen as the non-relativistic limit of \eqref{intro_rel}. A natural problem is then to study the asympotic behaviour of the solutions given by Theorem \ref{thmain} in the non-relativistic limit, in order to investigate whether they converge to solutions of \eqref{eq-class}. This is a very delicate question and it seems to be difficult to give a general answer. We refer to the study of the model problem of Section \ref{sec4} (see Remark \ref{rem-nonrel1}) for a partial result in this direction when $N=1$ and $W\equiv 0$.
\end{remark}

\begin{remark}\label{remstrong}
\rm
As well known since the pioneering paper by Gordon \cite{Gor75}, a condition on the strength of the singularity is typically needed when applying variational methods to problems with singular potentials. In the classic (i.e., non-relativistic) case, for potentials of the type \eqref{potential} the usual ``strong force'' condition reads as $\alpha \geq 2$; in our relativistic setting, instead, it is enough to assume the weaker condition $\alpha > 1$. This is a consequence of the fact that the expression $Z_h$, defined in \eqref{def-Z2}  and appearing in the Maupertuis functional \eqref{mau-func2}, behaves near the singularity as the square of the potential $V$: as a consequence, the potential $Z_h$ satisfies the property \eqref{p-Z1} in the proof of Theorem \ref{thmain}. Notice, however, that for \eqref{p-Z1} to be true $\alpha\geq 1$ would be enough, while Theorem \ref{thmain} does not cover the (physically relevant) Newtonian case $\alpha = 1$. 
The reason for this is that the further property \eqref{p-Z2}, required as well in the proof of Theorem \ref{thmain}, is valid only when $\alpha > 1$.
\end{remark}

\begin{remark}
\rm
As observed in Remark \ref{remhomotopy}, Theorem \ref{thmain} in the case $N=1$ just provides the existence of \emph{one} periodic solution of energy $\mathfrak{E}$, for any $\mathfrak{E} > mc^2$.
Actually, this conclusion is a direct consequence of \cite[Theorem 15.1]{ACZ93} (based on the result in \cite{Pis93}).
Indeed, as recalled in the proof of Proposition \ref{pr-maup}, critical points of the Maupertuis functional give rise (via a linear reparameterization of time) to periodic solutions of the usual second order system
$$
z'' = \nabla Z_h(z)
$$
with energy $2hm > 0$. Since, as observed in the proof of Theorem \ref{thmain}, the potential $Z_h$ satisfies assumptions \eqref{p-Z1}- \eqref{p-Z2}, and, moreover, $Z_h(u) > 0$ and $\nabla Z_h(u) \to 0$ for $\vert u \vert \to +\infty$, all the assumptions 
required in \cite[Theorem 15.1]{ACZ93} are satisfied (by a translation, we can assume in this case $\sigma_1 = 0$), and thus the conclusion follows. It is worth noticing that the arguments used in the proof of \cite[Theorem 15.1]{ACZ93}, relying on min-max critical point theory, are suitable also to treat the analogous problem in arbitrary space dimension.
\end{remark}


\begin{remark}\label{remsmooth} \rm Let us observe that by means of a direct application of standard results it is possible to prove the existence of periodic solutions with fixed energy for the equation \eqref{2.1} when $V$ is a globally defined smooth potential. Indeed, from \cite[Theorem 1.4]{Ben83} applied to the energy functional $\mathcal{E}_h$ (see also Theorem 1.1 in \cite{ACZ93}) we immediately deduce the validity of the following result:
	\begin{quotation}
	\noindent
		\textbf{Theorem 1.} \textit{Let $V:\R^n\to \R$, $V\in C^1(\R^n)$, $h\in \R$ and $Z_h$ as in \eqref{def-Z}. Assume that at least one connected component of the set 
		\[
		\{x\in \R^n:\ Z_h(x)+2hm\geq 0\}
		\]
		is non-empty and compact and 
		\[
		\nabla Z_h(x)\neq 0,\quad \forall \ x\in Z_h^{-1}(-2hm).
		\]
		Then, there exists a periodic solution $q:[0,1]\to \R^n$ of the geodesic equation \eqref{EL-eqn}.}
	\end{quotation}
\noindent
Recalling Proposition \ref{pr-var}, Theorem 1 implies the existence of a periodic solution of \eqref{2.1} with relativistic energy $\mathfrak{E} = h+mc^2$.

The assumptions of Theorem 1 can be written in term of the original potential $V$: indeed, recalling \eqref{def-Z}, it is easy to see that
\[
\{x\in \R^n:\ Z_h(x)+2hm\geq 0\} = \{x\in \R^n:\ V(x)+h\geq 0\} \cup \{x\in \R^n:\ V(x)+h\leq -2mc^2\}
\]
and that
\[
\nabla Z_h(x)=\dfrac{2}{c^2}\, (V(x)+h+mc^2)\, \nabla V(x),\quad \forall \ x\in \R^n.
\]
Now, let us recall that solutions $x$ of \eqref{2.1} with relativistic energy $h+mc^2$ are confined in the Hill's region $\Omega_h$ defined in \eqref{eq-defomegah}. Hence, if there exists $h\in \R$ are such that the set
	\[
	\{x\in \R^n:\ V(x)+h\geq 0\}
	\]
is non-empty and compact and 
\[
\nabla V(x)\neq 0,\quad \forall \ x\in V^{-1}(-h),
\]
then the assumptions of Theorem 1 are satisfied and the existence of a periodic solution of \eqref{2.1} with energy $\mathfrak{E} = h+mc^2$ is proved.

Let us notice that when $\beta>1$ and $h>0$ this result applies in particular to the pure power potential $V:\R^n\to \R$ defined by
\[
V(x)=-\frac{|x|^\beta}{\beta},
\]
for every $x\in \R^n$. 
\end{remark}

\section{Remarks on the optimality of the condition $\mathfrak{E} > mc^2$}\label{sec4}

With the aim of establishing the sharpness of the assumption $\mathfrak{E} > mc^2$ (equivalently, $h > 0$) in Theorem \ref{thmain}, in this section we focus on the model equation
\begin{equation}
\label{eq-4.1}
\frac{\rm d}{{\rm d}t}\left(\frac{m\dot{x}}{\sqrt{1-|\dot{x}|^2/c^2}}\right)=-\frac{\kappa x}{|x|^{\a+2}}, \qquad x \in \mathbb{R}^2 \setminus \{0\},
\end{equation}
with $\kappa > 0$ and $\alpha > 1$, which corresponds to \eqref{eq-ncariche} with a single centre at the origin (that is, $N=1$ and $\sigma_1=0$) and $W\equiv 0$. Let us recall that the energy of solutions $x$ of \eqref{eq-4.1} in this case is given by
\begin{equation}
\label{en}
E(x,\dot{x})=\frac{mc^2}{\sqrt{1-|\dot{x}|^2/c^2}}-\frac{\kappa}{\alpha|x|^{\alpha}}.
\end{equation}
Obviously, equation \eqref{eq-4.1} can have two different types of periodic solutions: the circular ones and the non-circular ones. 
The following result holds true.

\begin{proposition}\label{prop-esist} Let $\alpha > 1$. Then:
	\begin{enumerate}
		\item equation \eqref{eq-4.1} has circular periodic solutions of energy $\mathfrak{E}$ if and only if
		\begin{equation} \label{eq-suffcirc}
		\mathfrak{E} >\eta(\alpha),
		\end{equation} 
		where $\eta: (1,+\infty)\to \R$ is defined by
		\[
		\eta(\alpha)=\left\{
		\begin{array}{ll}
		\displaystyle 2mc^2\, \dfrac{\sqrt{\alpha-1}}{\alpha} &\quad\mbox{if $\alpha \in (1,2)$,}\\
		&\\
		mc^2&\quad\mbox{if $\alpha \geq 2$.}
		\end{array}
		\right.
		\]
		\item equation \eqref{eq-4.1} does not admit non-circular periodic solutions if $\alpha \geq 2$.
	\end{enumerate}
\end{proposition}

\begin{remark}\label{rem-finale}  
\rm Proposition \ref{prop-esist} immediately implies that the condition $\mathfrak{E}>mc^2$ assumed in Section \ref{sec3} is optimal when $\alpha \geq 2$. Let us also observe that, in this case, the arguments used along the proof allow us to conclude that for every $\mathfrak{E} > mc^2$ there exists a unique (up to time-translation and time-inversion) periodic solution of energy $\mathfrak{E}$, which, thus, is necessarily the solution provided by Theorem \ref{thmain} (recall also Remark \ref{remhomotopy}). As a consequence, we can infer that circular solutions are local minimizers for the Maupertuis functional whenever $\alpha \geq 2$.
Notice that for $\alpha \in (1,2)$, instead, we cannot obtain the same conclusion since non-circular solutions could exist.

Moreover, again for $\alpha \in (1,2)$, we do not know if the existence of periodic solutions when $2mc^2 \sqrt{1-\alpha}/\alpha<\mathfrak{E} \leq mc^2$ can be ensured also for the more general problem \eqref{eq-ncariche}.
\end{remark}

\begin{proof}[Proof of Proposition \ref{prop-esist}] 
	
\noindent
1.	The study of the existence of circular solutions of \eqref{eq-4.1} can be carried out by passing to polar coordinates, writing 
\[
x(t)=re^{i\omega(t)},\quad \forall \ t\in \R,
\]
where $r>0$. It is possible to see that $\dot{\omega}$ is constant. Indeed, for solutions $x$ of \eqref{eq-4.1} the relativistic angular momentum
\[
x\wedge \frac{m\dot{x}}{\sqrt{1-|\dot{x}|^2/c^2}}
\] 
(see also \eqref{eq-momenti} and the discussion at the beginning of the proof of the second statement) is conserved. Hence  
\[
\frac{mr^2 \dot{\omega}}{\sqrt{1-|\dot{\omega}|^2 r^2/c^2}}
\]
is constant, thus implying that $\dot{\omega}$ is constant. We can then assume that
	\[
	x(t)=re^{i\omega t},\quad \forall \ t\in \R,
	\]
	for some $\omega \in \R$. We immediately deduce that
	\[
\frac{\rm d}{{\rm d}t}\left(\frac{m\dot{x}}{\sqrt{1-|\dot{x}|^2/c^2}}\right)=-\frac{cm\omega^2r}{\sqrt{c^2-\omega^2r^2}} \, e^{i\omega t},
\]
for every $t\in \R$. Hence, $x$ is a solution of \eqref{eq-4.1} if and only if
	\begin{equation}
	\label{relation}
	-\frac{cm\omega^2r}{\sqrt{c^2-\omega^2r^2}}\, e^{i\omega t}=-\frac{\kappa}{r^{\alpha+1}}\, e^{i\omega t}, 
	\end{equation}
i.e. if and only if $r$ and $\omega$ are related by 
	\[
	\omega^2=\frac{-\kappa^2+\kappa\sqrt{\kappa^2+4c^4m^2r^{2\alpha}}}{2m^2c^2r^{2\alpha+2}}. 
	\]
From this condition we deduce that the energy given in \eqref{en} can be rewritten as
	\begin{equation} \label{eq-energiaridotta}
	\begin{split} 
E(x,\dot{x})=\dfrac{c^2\kappa}{r^{\alpha+2}\omega^2}-\dfrac{\kappa}{\alpha r^{\alpha}}=
	\kappa\left(\dfrac{2m^2c^4r^{\alpha}}{-\kappa^2+\kappa\sqrt{\kappa^2+4m^2c^4r^{2\alpha}}}-\dfrac{1}{\alpha r^{\alpha}}\right)=\kappa \, f(r^\alpha),
	\end{split}
	\end{equation}
where 
\[
f(t)=\frac{2m^2c^4t}{-\kappa^2+\kappa\sqrt{\kappa^2+4m^2c^4t^{2}}}-\frac{1}{\alpha t},\quad \forall \ t>0.
\]
We now study the admissible values of this energy, depending on the radius $r>0$. It is immediate so see that
\begin{equation} \label{eq-liminf}
	\lim_{t\rightarrow+\infty}f(t)=\frac{mc^2}{\kappa}.
	\end{equation}
	On the other hand, from the Taylor expansion
	\[
	\sqrt{\kappa^2+4m^2c^4t^2}=\kappa\left(1+2\frac{m^2c^4}{\kappa^2}t^2-2\frac{m^4c^8}{\kappa^4}t^4\right)+o(t^4),\quad t\to 0,
	\]
we deduce
	\[
	f(t)=\dfrac{2m^2c^4t}{2m^2c^4 t^2-2\frac{m^4c^8}{\kappa^2}t^4+o(t^4)}-\frac{1}{\alpha t}=\left(1-\frac{1}{\alpha}\right)\frac{1}{t}+o\left(\frac{1}{t}\right),\quad t\to 0. 
	\]
	Therefore, since $\alpha>1$, we obtain
	\begin{equation} \label{eq-limzero}
	\lim_{t\rightarrow 0^+} f(t)=+\infty. 
	\end{equation}
As far as the monotonicity of $f$ is concerned, we have
	\begin{equation} \label{eq-pr31}
	f'(t)=\dfrac{2m^2c^4\left(-\kappa^2+k\sqrt{\kappa^2+4m^2c^4t^2}\right)-8m^4c^8\kappa t^2\sqrt{\kappa^2+4m^2c^4t^2}}{\left(-\kappa^2+\kappa\sqrt{\kappa^2+4m^2c^4t^2}\right)^2}+\dfrac{1}{\alpha t^2},
	\end{equation}
	for every $t>0$. Now, let us define
	\[
	g(t)=\sqrt{\kappa^2+4m^2c^4t^2},
	\]
	for every $t>0$. From \eqref{eq-pr31} we infer that
	\begin{align*}
	f'(t)&=\dfrac{2m^2c^4[(-\kappa^2+\kappa g(t))g(t)-4m^2c^4\kappa t^2]}{g(t)(-\kappa^2+g(t))^2}+\dfrac{1}{\alpha t^2}\\&=-\dfrac{2m^2c^4 \kappa^2}{g(t)(g(t)-\kappa)}+\dfrac{1}{\alpha t^2}=\dfrac{-2\alpha t^2m^2c^4\kappa^2+g(t)(g(t)-\kappa)}{\alpha t^2 g(t)(g(t)-\kappa)},
	\end{align*}
	for every $t>0$. A simple computation shows that 
	\begin{equation} \label{eq-pr41}
f'(t)\geq 0 \quad \Longleftrightarrow \quad	\kappa g(t)\leq \kappa^2+2(2-\alpha)m^2c^4t^2 \quad \Longleftrightarrow \quad	\kappa g(t)- \kappa^2 \leq 2(2-\alpha)m^2c^4t^2.
	\end{equation}
	
Now, if $\alpha \geq 2$, observing that $\kappa g(t)- \kappa^2>0$, for every $t>0$, we deduce that $f'(t)< 0$, for every $t>0$; taking into account \eqref{eq-energiaridotta}, \eqref{eq-liminf} and \eqref{eq-limzero}, we can conclude that in this case the range of the energy is $(mc^2,+\infty)$. As a consequence, for every $\mathfrak{E} \in (mc^2,+\infty)$ there exists $r_\mathfrak{E} >0$ such that 
\[
E(x,\dot{x})=\mathfrak{E},
\]
where $x(t)=r_\mathfrak{E}  e^{i\omega_\mathfrak{E}  t}$, with $\omega_\mathfrak{E} $ as in \eqref{relation}. This proves the validity of \eqref{eq-suffcirc} for $\alpha\geq 2$.	

On the other hand, if $\alpha \in (1,2)$, from \eqref{eq-pr41} we deduce that
	\[
	\begin{array}{ll}
	f'(t)\geq 0 \quad & \displaystyle \Longleftrightarrow \kappa ^2g^2(t)\leq (\kappa^2+2(2-\alpha)m^2c^4t^2)^2\\
	&\\
	&\displaystyle \Longleftrightarrow (2-\alpha)^2m^2c^4t^2+(1-\alpha)\kappa ^2\geq 0. 
	\end{array}
	\]
Hence, we have
	\[
	\begin{array}{ll}
	f'(t)\geq 0 \displaystyle \Longleftrightarrow t\geq\frac{\kappa}{mc^2}\frac{\sqrt{\alpha-1}}{(2-\alpha)}=:t_{\min}. 
	\end{array}
	\]
	As a consequence, $f$ has the unique mininum point $t_{\min}$ with corresponding minimum
	\[
	f(t_{\min})=\dfrac{2m^2c^4t_{\min}}{-\kappa^2+\kappa g(t_{\min})}-\dfrac{1}{\alpha t_{\min}}=\dfrac{1}{t_{\min}}\left(\dfrac{1}{2-\alpha}-\frac{1}{\alpha}\right)=\dfrac{2mc^2}{\kappa}\, \dfrac{\sqrt{\alpha-1}}{\alpha}
	\]
	Recalling \eqref{eq-energiaridotta}, we conclude that the minimum value of the energy is
	\[
2mc^2\, \dfrac{\sqrt{\alpha-1}}{\alpha}.
	\]
	Arguing as above, taking into account \eqref{eq-limzero}, we conclude that there exists a circular solution with energy $\mathfrak{E} $ if and only if $\mathfrak{E}  >2mc^2 \sqrt{1-\alpha}/\alpha$. 

\smallskip
\noindent
2. In order to study non-circular periodic solutions, we follow the approach of \cite[Section 2.2]{BDF21} (which deals with the Kepler/Coulomb case $\alpha=1$). Let us first observe that, according to the discussion of Section 2, equation \eqref{eq-4.1} has a Hamiltonian structure. The Hamiltonian given in \eqref{hamilt} in this case is  
\[
H(x,p)=mc^2\sqrt{1+\frac{|p|^2}{m^2c^2}}-\dfrac{\kappa}{\alpha |x|^\alpha},
\]
where $p$ is defined in \eqref{dual}. We introduce the angular momentum 
\begin{equation} \label{eq-momenti}
\mathcal{L}=\langle x,Jp\rangle,\quad\hbox{where}\;\, J = 
\begin{pmatrix}
0 &\; 1 \\
-1 &\; 0 \\
\end{pmatrix}.
\end{equation}
Now, let us define
\begin{equation*}
\Upsilon=(0,+\infty)\times \mathbb{T}^1 \times \mathbb{R}^2
\end{equation*}
and the symplectic diffeomorphism
\[
\Psi \colon \Upsilon \to (\mathbb{R}^2\setminus\{0\})\times \mathbb{R}^2, \qquad (r,\vartheta,l,\Xi)\mapsto (x,p), 
\]
given by 
\[
x=re^{i\vartheta}, \qquad p=le^{i\vartheta}+\dfrac{\Xi}{r}ie^{i\vartheta}.
\]
In the new variables the Hamiltonian and the angular momentum, which are first integrals, can be written as 
\[
\mathcal{L}_0(r,\vartheta,l,\Xi)=\Xi
\]
and
\[
H_0(r,\vartheta,l,\Xi)=mc^2\sqrt{1+\dfrac{l^2+\Xi^2/r^2}{m^2c^2}}-\dfrac{\alpha}{r}.
\]
Hence, every solution $(r,l,\vartheta,\Xi)=(r(t),l(t),\vartheta(t),\Xi(t))$ of the Hamiltonian system satisfies
\begin{equation*}
mc^2\sqrt{1+\dfrac{l(t)^2+L^2/(r(t))^2}{m^2c^2}}-\dfrac{\kappa}{\alpha r(t)}\equiv \mathfrak{E},
\end{equation*}
for some $\mathfrak{E}, L \in \mathbb{R}$. This relation can be written as 
\begin{equation} \label{eq-pr35}
\ell^2=\frac{1}{c^2}\left(\frac{\kappa^2}{\alpha^2r^{2\alpha}}-\frac{c^2L^2}{r^2}+2\mathfrak{E}\frac{\kappa}{\alpha r^{\alpha}}+\mathfrak{E}^2-m^2c^4\right).
\end{equation}
Let us now fix $L>0$ and a value $\mathfrak{E}\in \R$ of the energy and define
\[
\Phi_{\mathfrak{E},L}(r)=\frac{1}{c^2}\left(\frac{\kappa^2}{\alpha^2r^{2\alpha}}-\frac{c^2L^2}{r^2}+2\mathfrak{E} \frac{\kappa}{\alpha r^{\alpha}}+\mathfrak{E}^2-m^2c^4\right),
\]
for every $r>0$. According to \eqref{eq-pr35}, non-constant trajectories in the phase-plane $(r,\ell)$ are confined in the region $\{ (r,\ell):\ \Phi_{\mathfrak{E},L}(r)\geq 0\}$. Moreover, closed orbits exist if and only if $\Phi_{\mathfrak{E},L}$ has two consecutive zeros and it is striclty positive between them. This cannot occur when $\alpha \geq 2$, as it is clear by studying the behaviour of the function $\Phi_{\mathfrak{E},L}$.

Indeed, assume first that $\alpha = 2$. In this case the function  $\Phi_{\mathfrak{E},L}$ can be written has
\[
\Phi_{\mathfrak{E},L}(r)=\frac{1}{c^2 r^4}\, P_2(r^2), \quad r>0,
\]
where $P_2$ is the second order polynomial defined by
\[
P_2(x)=\frac{\kappa^2}{\alpha^2}-\left(c^2L^2+2\mathfrak{E} \frac{\kappa}{\alpha}\right)\, x+(\mathfrak{E}^2-m^2c^4)\, x^2,
\]
for every $x\geq 0$. Since $P_2(0)>0$, the polynomial $P_2$ (and then $\Phi_{\mathfrak{E},L}$) cannot have two consecutive positive zeros, being positive between them. Hence, no bounded non-costant orbits exist.

Now, let us consider the case $\alpha >2$. We first observe that
\begin{equation}
\label{phi}
\lim_{r\rightarrow 0^+}\Phi_{\mathfrak{E},L}(r)=+\infty, \quad \lim_{r\rightarrow +\infty}\Phi_{\mathfrak{E},L}(r)=\dfrac{\mathfrak{E}^2-m^2c^4}{c^2}.
\end{equation}
and
\begin{equation} \label{eq-derphi}
\Phi_{\mathfrak{E},L}'(r)=\dfrac{2}{c^2r^{2\alpha+1}}\left(c^2L^2r^{2\alpha-2}-\dfrac{\kappa^2}{\alpha}-2\mathfrak{E} \kappa r^\alpha\right)=\dfrac{2}{c^2r^{2\alpha+1}}\, \widetilde{\Phi}_{\mathfrak{E},L}(r),
\end{equation}
where $\widetilde{\Phi}_{\mathfrak{E},L}:(0,+\infty)\to \R$ is defined by
\[
\widetilde{\Phi}_{\mathfrak{E},L}(r)=c^2L^2r^{2\alpha-2}-\dfrac{\kappa^2}{\alpha}-2\mathfrak{E} \kappa r^\alpha,
\]
for every $r>0$. It is immediate to see that
\begin{equation}
\label{tilde1}
\lim_{r\rightarrow 0^+}\widetilde{\Phi}_{\mathfrak{E},L}(r)=-\dfrac{\kappa^2}{\alpha}\quad\hbox{and}\quad \lim_{r\rightarrow +\infty}\widetilde{\Phi}_{\mathfrak{E},L}(r)=+\infty.
\end{equation}
Moreover, we have
\[
\widetilde{\Phi}'_{\mathfrak{E},L}(r)=2r^{2\alpha-3}\, \left( (\alpha-1)c^2L^2-\alpha \mathfrak{E}\kappa r^{2-\alpha}\right),
\]
for every $r>0$. As a consequence, we obtain
\begin{equation}
\label{tilde3}
\begin{array}{ll}
\widetilde{\Phi}'_{\mathfrak{E},L}(r) > 0,\quad \forall \ r> 0 & \; \text{ if } \mathfrak{E}\leq 0,
\\
\widetilde{\Phi}'_{\mathfrak{E},L}(r)> 0\quad \Longleftrightarrow \quad 
\displaystyle r > \left(\dfrac{\alpha-1}{\alpha}\dfrac{c^2L^2}{\mathfrak{E}\kappa}\right)^  {1/(2-\alpha)} & \; \text{ if } \mathfrak{E} > 0.
\end{array}
\end{equation}
Therefore, from \eqref{tilde1} and \eqref{tilde3} we deduce that, in both cases, $\widetilde{\Phi}_{\mathfrak{E},L}$ has exactly one zero. Taking into account \eqref{phi} and \eqref{eq-derphi}, we conclude that ${\Phi}_{\mathfrak{E},L}$ has exactly one critical point, which is a global minimum point. This proves that ${\Phi}_{\mathfrak{E},L}$ cannot have two consecutive zeros being strictly positive between them. Hence, no bounded non-costant orbits exist.
\end{proof}

\begin{remark}\label{rem_alfa} 
\rm
We point out that in the Keplerian/Coulombian case $\alpha=1$ periodic solutions of energy $\mathfrak{E}$ exist if and only if 
$0<\mathfrak{E}<mc^2$. Indeed, in \cite{BDF21} it is proved that this condition is necessary and sufficient for the existence of non-circular periodic solutions. On the other hand, using the notation in the proof of the first statement of Proposition \ref{prop-esist}, 
it can be checked that the function $f$ is increasing, with $\lim_{t \to 0^+} f(t) = 0$ and $\lim_{t \to +\infty} f(t) = mc^2/\kappa$: therefore, the existence of circular periodic solutions is guaranteed if and only if $0<\mathfrak{E}<mc^2$. The value $\alpha =1$ is then in some sense a threshold for the range of admissible energies. 
\end{remark}

\begin{remark}\label{rem-nonrel} 
\rm
We conclude this section with the asymptotic study for $c\to +\infty$ of the circular solutions of energy $\mathfrak{E}>mc^2$ (i.e. $h>0$) whose existence has been proved in Proposition \ref{prop-esist}. We refer to Remark \ref{rem-nonrel1} for the motivation and the general discussion.
In the case of the model problem studied in this section, the limit problem when $c\to +\infty$ is 
\[
m \ddot x=-\frac{\kappa x}{|x|^{\a+2}}.
\] 
It is well known that circular solutions of energy $h>0$ of this problem exist if and only if $\alpha>2$ and that in this case their radius $R_h$ is given by
\begin{equation} \label{eq-raggiocircclass}
R^\alpha_h =\dfrac{\kappa}{h}\, \dfrac{\alpha -2}{2\alpha}.
\end{equation}
On the other hand, we recall from the proof of Proposition \ref{prop-esist} that the radius $r=r_h(c)$ of these circular solutions satisfies the relation
\begin{equation} \label{eq-nuovaeq}
\varphi(r,c)=h+mc^2,
\end{equation}
where $\varphi:(0,+\infty)\times (0,+\infty)\to \R$ is defined by
\begin{equation} \label{eq-defvarphi}
\varphi (r,c)=\kappa\left(\dfrac{2m^2c^4r^{\alpha}}{-\kappa^2+\kappa\sqrt{\kappa^2+4m^2c^4r^{2\alpha}}}-\dfrac{1}{\alpha r^{\alpha}}\right),\quad \forall \ (r,c)\in (0,+\infty)\times (0,+\infty)
\end{equation}
(cf. \eqref{eq-energiaridotta}). More precisely, the discussion in the proof of Proposition \ref{prop-esist} shows that for every $h>0$ equation \eqref{eq-nuovaeq} can be uniquely solved for $r$ as a  function of $c$, and
\begin{equation}\label{eq-seg1}
\partial_r \varphi(r,c)<0,\quad \text{ for every $(r,c)$ such that } \varphi(r,c) = h + mc^2.
\end{equation}
Let us notice, from a different point of view, that the function $r_h$ can be seen as implicitely defined by equation \eqref{eq-nuovaeq}. The implicit function theorem implies then that $r_h$ is differentiable. A simple computation proves that
\begin{align*}
\partial_c \varphi(r,c) & = \displaystyle -\dfrac{4mc}{\sqrt{\kappa^2+4m^2c^4 r^{2\alpha}}}\, \dfrac{1}{(-\kappa+\sqrt{\kappa^2+4m^2c^4 r^{2\alpha}})^2}\cdot \\
&\quad \displaystyle \cdot \Big(\sqrt{\kappa^2+4m^2c^4 r^{2\alpha}}\left(2m^2c^4 r^{2\alpha}+2\kappa mc^2 r^{\alpha}+\kappa^2 \right)\\
&\quad \displaystyle  -4m^3 c^6 r^{3\alpha} -4m^2 c^4 \kappa r^{2\alpha} -2mc^2\kappa^2 r^{\alpha}-\kappa^3\Big) ,
\end{align*}
for every $(r,c)\in (0,+\infty)\times (0,+\infty)$.	It is then easy to see that
\begin{equation}\label{eq-seg2}
\partial_c \varphi(r,c)<0,\quad \forall \ (r,c)\in (0,+\infty)\times (0,+\infty).
\end{equation}
Hence, differentiating \eqref{eq-nuovaeq} with respect to $c$ and taking into account \eqref{eq-seg1} and \eqref{eq-seg2}, we deduce that
\[
r'_h(c)=\dfrac{2mc - \partial_c \varphi(r_h(c),c)}{\partial_r \varphi(r_h(c),c)}<0,\quad \forall \ c>0.
\]
Ad a consequence, $r_h$ decreases and there exists
\[
\lim_{c\to +\infty} r_h(c)=r_h(\infty)\geq 0.
\]
We now show that $r_h(\infty)=R_h$, with $R_h$ as in \eqref{eq-raggiocircclass}, when $\alpha >2$ and that $r_h(\infty)=0$ when $\alpha \in (1,2]$, thus proving that for $h>0$ there is convergence to the non-relativistic circular solutions when $\alpha >2$, while relativistic circular solutions disappear in the non-relativistic limit when $\alpha \in (1,2]$. To this end, let us observe that, using \eqref{eq-defvarphi}, equation \eqref{eq-nuovaeq} can be written as
\[
\dfrac{2m^2c^4r^{2\alpha}}{-\kappa+\sqrt{\kappa^2+4m^2c^4r^{2\alpha}}}-mc^2 r^{\alpha}=hr^{\alpha} \, +\dfrac{\kappa}{\alpha},
\]
i.e. 
\begin{equation}\label{eq-ultimaeq}
\psi (mc^2 r^{\alpha})=\Theta(r^\alpha),
\end{equation}
where $\psi, \Theta:(0,+\infty)\to \R$ are defined by
\[
\psi(x)=\dfrac{2x^2}{-\kappa+\sqrt{\kappa^2+4x^2}}-x,\quad \Theta(x)=hx+\dfrac{\kappa}{\alpha},
\]
for every $x>0$, respectively. It is easy to prove that $\psi$ is a decreasing function with range $(\kappa/2,\kappa)$. 

Now, from \eqref{eq-ultimaeq} we deduce that 
\begin{equation} \label{eq-pr88}
\psi (mc^2 r^{\alpha}_h(c))=\Theta(r^\alpha_h(c)),
\end{equation}
for every $c>1$. Let us notice that if $r_h(\infty)>0$ then $mc^2 r_h(c)\to +\infty$ for $c\to +\infty$; as a consequence, passing to the limit in \eqref{eq-pr88} and recalling the definition of $\Theta$, we obtain
\[
\dfrac{\kappa}{2}=h r^\alpha_h(\infty) +\dfrac{\kappa}{\alpha},
\]
i.e. 
\begin{equation} \label{eq-pr89}
\kappa \left(\dfrac{1}{2}-\dfrac{1}{\alpha}\right)=h r^\alpha_h(\infty),
\end{equation}
which is impossible if $h>0$ and $\alpha \in (1,2]$. This shows that in the case $\alpha \in (1,2]$ the limiting radius is $0$. On the other hand, if $\alpha >2$ and $r_h(\infty)=0$, then
\[
\lim_{c\to +\infty} \Theta (r_h^\alpha (c))=\dfrac{\kappa}{\alpha}<\dfrac{\kappa}{2}.
\]
Hence, there exists $c_0>0$ such that
\[
\Theta (r_h^\alpha (c))<\dfrac{\kappa}{2},\quad \forall \ c>c_0.
\]
Thus, for every $c>c_0$, the value $\Theta (r_h^\alpha (c))$ does not belong to the range of $\psi$, contradicting \eqref{eq-pr88}. We then conclude that if $\alpha >2$ necessarily $r_h(\infty)>0$ and, as a consequence, \eqref{eq-pr89} holds true. This implies that $r_h(\infty)=R_h$ and the circular relativistic solutions converge to the circular classical solution of energy $h>0$.
\end{remark}

\bigskip
\noindent
\textbf{Data availability statement.} Data sharing not applicable to this article as no datasets were generated or analysed during the current study.

\bibliographystyle{plain}
\bibliography{Bo-Da-MH}

\end{document}